\documentclass[11pt, reqno]{amsart}
\usepackage{enumitem}
\usepackage{a4wide}
\usepackage{amsmath,amssymb,amsfonts,amsthm}
\usepackage{accents}
\usepackage{xcolor}
\usepackage{xfrac}

\usepackage{tikz}
\usetikzlibrary{decorations.pathreplacing}
\usepackage{cancel}

\usepackage[colorlinks=true,urlcolor=blue,
citecolor=red,linkcolor=blue,linktocpage,pdfpagelabels,
bookmarksnumbered,bookmarksopen]{hyperref}

\theoremstyle{plain}

\newtheorem{theorem}{Theorem}[section]
\newtheorem{lemma}[theorem]{Lemma}
\newtheorem{proposition}[theorem]{Proposition}
\newtheorem{corollary}[theorem]{Corollary}
\newtheorem{remark}[theorem]{Remark}
\newtheorem{definition}[theorem]{Definition}
\theoremstyle{definition}
\numberwithin{equation}{section}

\def\R{\mathbb{R}}
\def\SSS{\mathbb{S}}

\def\Z{\mathbb{Z}}
\def\N{\mathbb{N}}
\def\dist{\textup{dist}}

\def\H{\mathcal{H}}

\def\T{\mathbb{T}}

\def\h{\mathrm{H}}

\renewcommand{\div}{\mathrm{div}}

\newcommand{\eps}{\varepsilon}
\renewcommand{\H}{\mathcal H}

\newcommand{\de}{\partial}
\newcommand{\pa}{\partial}
\newcommand{\e}{F}
\newcommand{\cW}{\mathcal{W}}

\newcommand{\medint}{-\kern -,375cm\int}
\newcommand{\medintinrigo}{-\kern -,315cm\int}
\def\ringg#1{\accentset{\circ}{#1}}

\def\beq{\begin{equation}}
\def\eeq{\end{equation}}

\title[Quantitative Alexandrov and applications]{A sharp quantitative Alexandrov inequality and applications to volume preserving geometric~flows~in~3D}

\author{Vesa Julin}\address{Matematiikan ja Tilastotieteen Laitos, Jyv\"askyl\"an Yliopisto, Finland}\email{vesa.julin@jyu.fi}

\author{Massimiliano Morini}\address{Dipartimento di Scienze Matematiche Fisiche e Informatiche, Universit\`a di Parma, Italy}\email{massimiliano.morini@unipr.it}

\author{Francesca Oronzio}\address{Institutionen f\"{o}r Matematik, Kungliga Tekniska H\"{o}gskolan, Sweden} \email{oronzio@kth.se}

\author{Emanuele Spadaro}\address{ Dip. di Matematica, Univ. Roma-I ''La Sapienza'' Roma, Italy}\email{spadaro@mat.uniroma1.it}

\begin{document}

\begin{abstract} 
We study the asymptotic behavior of the volume preserving mean curvature and the Mullins-Sekerka flat ﬂow in three dimensional space. Motivated by this  we establish a 3D sharp quantitative version of the Alexandrov inequality for $C^2$-regular sets with a perimeter bound.
\end{abstract}

\maketitle

\section{Introduction}

In this paper we continue our study on the long-time behavior of globally well-defined  weak solutions of two physically relevant volume preserving geometric ﬂows in three dimensions: the volume preserving mean curvature and the Mullins-Sekerka flat ﬂow. Our starting point is the work \cite{JN}, which implies the qualitative convergence of the volume preserving mean curvature flat flow in $\R^3$ to a union of disjoint balls, up to translation of the components. In our previous work \cite{JuMoPoSpa} on the two-dimensional case, we observe that the full quantitative convergence is related to a sharp quantitative version of the Alexandrov theorem, which is  a purely geometric inequality. In \cite{JuMoPoSpa} we prove a sharp quantitative Alexandrov theorem and show that it not only implies the full convergence of both the area preserving mean curvature and the Mullins-Sekerka flat ﬂow, but also gives the exponential rate of convergence, in the planar case. Our main results in this paper are the three dimensional counterpart of both this inequality and its consequences for the asymptotics of the aforementioned  geometric flat flows.

\subsection{Quantitative Alexandrov inequality}

We denote by $\mathrm{H}_E$ the mean curvature of the hypersurface $\partial E$ (with the sign convention that $\mathrm{H}_E\geq 0$ if $E$ is convex, see the next section for details), $P(E)$ for the perimeter and set 
\[
\mathrm{\bar H}_E := \frac{1}{P(E)}\int_{\partial E} \mathrm{H}_E\, d\mathcal{H}^2.
\]
Here is  our first main results.
\begin{theorem}\label{t:LS}
For every $\delta_0>0$ there exists $C=C(\delta_0)>0$ such that for every $C^2$-regular set $E\subset \R^3$ with 
\begin{equation}\label{e.condition}
|E|=|B_1| \quad \text{and}\quad P(E) \leq 4 \pi \, \sqrt[3]{2} -\delta_0 
\end{equation}
 it holds 
\begin{equation}\label{the-inequality}
P(E)-P(B_1)\leq C \, \big\|\mathrm{H}_E-\mathrm{\bar H}_E\big\|_{L^2(\partial E)}^2.
\end{equation}
Moreover, there exists $\delta_1>0$ such that, if in addition $\pa E\in C^\infty$ and $\|\mathrm{H}_E-\mathrm{\bar H}_E\|_{L^2(\partial E)}\leq \delta_1$, then $\partial E$ is diffeomorphic to the standard sphere $\mathbb{S}^2$.
\end{theorem}

Let us make some comments on the result. The inequality \eqref{the-inequality} is a 
Łojasiewicz--Simon inequality for the perimeter functional with optimal exponent $2$: indeed, \eqref{the-inequality} bounds the perimeter with the sharp power of the $L^2$-norm of its first variation, which is the mean-curvature.
We also call \eqref{the-inequality}  a quantitative Alexandrov inequality because when the right hand side is zero, i.e., when $\pa E$ has constant mean curvature, then $P(E) \leq P(B_1)$, which by the isoperimetric inequality means that $E$ is the ball. This is the famous Alexandrov theorem and \eqref{the-inequality}  is a sharp quantitative version of it, in the sense that the exponent two on the right hand side is optimal. Indeed,  \eqref{the-inequality}  is false for any exponent higher than two, see \cite[Remark 2.2]{JuMoPoSpa}. 

There  has been an increasing interest  on generalizations and quantifications of the Alexandrov theorem in recent years. We refer to \cite{Ci} for an overview of this challenging problem, and mention the works \cite{DM, DMMN, RKS} on the  characterization of critical sets of the isoperimetric problem and \cite{CM,CV, KM} on quantification of the Alexandrov theorem. The main issue in the problem in Theorem \ref{t:LS} is that we do not have any a priori estimates, such as curvature bounds, for the set $E$ other than the perimeter bound in \eqref{e.condition}, which prevents the bubbling phenomenon.  Moreover, we measure the distance of the mean curvature to a constant with respect to $L^2$-norm, which in three dimension is not enough to apply Allard regularity theorem to deduce a priori regularity estimates for the set $E$. This makes the problem much more challenging than the two dimensional case. However, the inequality \eqref{the-inequality} is suitable for the study of the geometric equations, which was  our main motivation to consider the problem in the first place.  

We note that the right hand side \eqref{the-inequality} in  $\R^3$ is  special, since it is scaling invariant. The novelty of the paper is that we prove the  inequality \eqref{the-inequality} by means of results from conformal  geometry, by transforming the problem into a minimization problem involving  Canham-Helfrich energy functional. Indeed, using the estimates from \cite{JN} we are able to restate the problem using parametizations by weak immersions from $\SSS^2$ to $\R^3$ which allow us to use the results \cite{MoRibub, MonSch, RivLec} to obtain the existence and regularity of the minimizer of an auxiliary problem (see Proposition \ref{t.main}). We then derive uniform regularity estimates which enable us to prove \eqref{the-inequality}. 

The closest result to Theorem \ref{t:LS} in the literature is the quantitative Willmore inequality due to R\"oger-Sch\"aztle \cite{RS}, which roughly states that for $E$ as in Theorem \ref{t:LS} and $\partial E$ diffeomorphic to the sphere it holds 
\begin{equation}\label{eq:roger-schatzle}
P(E)-P(B_1)\leq C( \|\mathrm{H}_E \|_{L^2(\partial E)}^2 - 16 \pi).
\end{equation}
 However, the inequality \eqref{the-inequality} is stronger than  \eqref{eq:roger-schatzle} (in the sense that  \eqref{the-inequality} implies \eqref{eq:roger-schatzle}) and \eqref{eq:roger-schatzle} is not enough for the applications to geometric flows. Moreover, the proof of \eqref{eq:roger-schatzle} is completely different to \eqref{the-inequality}, as it relies on the quantitative  result on nearly umbilical sets due to De Lellis--M\"uller \cite{DLM}. 

\subsection{Application I: volume preserving MCF}
As we already mentioned, our main motivation to prove Theorem \ref{t:LS} is to study the asymptotic behavior of volume preserving geometric flows. We begin with the  {\em volume preserving mean curvature flow} and recall that  a smooth flow of sets $(E_t)_{t\in [0,T)}\subset \R^3$, for some $T>0$, is a solution to the volume preserving mean curvature flow if it satisfies
\begin{equation}\label{eq:VMCF}
V_t = -\mathrm{H}_{E_t} + \mathrm{\bar H}_{E_t}  \quad\text{on }\pa E_t\subset\R^3,
\end{equation}
where $V_t$ denotes the outer normal velocity. Such a geometric flow has been proposed in the physical literature as a model for coarsening phenomena. We refer to \cite{CRCT95, MuSe13,TW72,W61} for an introduction to the physical background.  For us the most important feature is that \eqref{eq:VMCF} can be seen  as a gradient flow of the perimeter with respect to a suitable (formal) $L^2$-type Riemannian structure and it preserves the volume.

In general smooth solutions of \eqref{eq:VMCF} may develop singularities  in finite time and therefore we need a suitable notion of weak solution  which is defined for all times. 
The main difference between  \eqref{eq:VMCF} and the mean curvature flows is that \eqref{eq:VMCF} is non-local and does not satisfy the comparison principle, and therefore it is not clear how to define a level set solution to it. A well established choice  for a weak solution of \eqref{eq:VMCF}  is the minimizing movement approach proposed for the mean curvature flow independently by Almgren, Taylor and Wang \cite{ATW} and by Luckhaus and Sturzenhecker \cite{LS}, and adapted to the volume constrained case in \cite{MSS}. Here we use the definition from \cite{Vesa} as it simplifies the analysis. We recall that the minimizing movement method is based on the gradient flow structure of the flow. To be more precise, one uses  implicit time-discretization and recursive minimization of a suitable incremental problem to construct a discrete-in-time approximation of a solution to the equation \eqref{eq:VMCF} and defines any of its cluster points, as the time step converges to zero, as \emph{a flat flow} solution of \eqref{eq:VMCF}. For the precise definition see Section \ref{sec:3bis}. By the results in  \cite{ATW, LS, MSS} the flat flow, starting from any bounded set of finite perimeter, always exists  and is H\"older continuous in the $L^1$-topology. Moreover by the result  in \cite{JN2} it coincides with the unique smooth solution as long as the latter exists. For other possible notions of weak solutions we refer to \cite{KK} for a viscosity solution for star-shaped sets, and the gradient flow calibration in \cite{Laux}. We also mention the recent works on the phase-field approximation to \eqref{eq:VMCF} via  the Allen-Cahn equation, which converges to a Brakke-type solution of \eqref{eq:VMCF} \cite{Taka1, Taka2}, to a distributional solution if the convergence is assumed to be strong in BV by \cite{LaSi18}, and the consistency is proven in \cite{KL}.

Since the flat flow is defined globally in time  we may study its asymptotics. Apart from our previous work \cite{JuMoPoSpa} that we already mentioned, the long-time convergence results in the literature  are confined to the case when the classical solution is defined for all times, e.g., when the initial set is convex  \cite{Hui}, or when the initial set is close to the ball \cite{ES} or to a general local minimum in the case of a flat torus \cite{Joonas, DDKK}. The starting point of our analysis is the result by first author and Niinikoski \cite{JN} on the qualitative asymptotic convergence of \eqref{eq:VMCF}, which is stated as follows:

\textit{Given any volume preserving flat flow $\{E(t)\}_{t\geq 0}$ in $\R^3$ starting from a bounded set of finite perimeter $E(0)\subset\R^3$ with volume $|E(0)| = v$, there exist $N \in \N$ and time-dependent points
\begin{equation}\label{e.centers}
x_1(t), \dots, x_N(t) \in \R^3, \quad |x_i(t) -x_j(t)| \geq 2r\quad\forall\;i\neq j,
\end{equation}
such that setting $r = \left( \frac{3 v}{4 \pi N}\right)^{1/3}$ and $F(t) = \bigcup_{i =1}^N B_r(x_i(t))$ we have that
\beq\label{asym}
\sup_{x\in E(t) \Delta F(t)} \dist(x, \pa F(t)) \to 0  \quad \text{as } \, t \to \infty.
\eeq
}
We remark that it is possible that the flow converges to a union of tangent balls \cite{FJM}, but  it is expected that this happens only in special cases, and that generically the limiting set is a unique ball. Based on this, we show that if in the above theorem, the set $F(t)$ is just one ball, or if the balls have positive distance to each other, then we may upgrade the above  convergence result to a full  quantitative exponential convergence  to a union of (independent of time) balls. In both cases, the crucial estimate is given by Theorem \ref{t:LS}.  We prove the following result.  
\begin{theorem}\label{thm2:MCF}
Let $\{E(t)\}_{t\geq 0}$, $N \in \N$ and the points  $x_1(t), \dots, x_N(t)$ in $\R^3$ be as \eqref{e.centers} above and assume in addition that there exists $\delta_1>0$ and $T_0$ such that 
\beq\label{separated}
|x_i(t) -x_j(t)| \geq 2r + \delta_1 \quad \text{ for all } \, t \geq T_0
\eeq
for $i \neq j$. Then the sets  $E(t)$ converge exponentially fast to  $F = \bigcup_{i =1}^N B_r(x_i)$, for some $x_1,\dots, x_N\in \R^3$ with  $|x_i -x_j| >2r$, $i \neq j$. To be more precise, we have
\[
\sup_{x\in E(t) \Delta F} \dist(x, \pa F) + |P(E(t)) - N 4 \pi r^2 | \leq C e^{-\frac{t}{C}}
\]  
for a constant $C>1$. 
\end{theorem}

In particular, if the set $F(t)$ in \eqref{asym} is just one ball, then by Theorem \ref{thm2:MCF} we have the full convergence of the flow. For example, if the initial set satisfies  the condition \eqref{e.condition}, i.e., $|E(0)| = |B_1|$ and $P(E(0))\leq 4 \pi \, \sqrt[3]{2} -\delta_0 $, or if  there is $T>0$ such that $E(t)$  satisfies  the  condition \eqref{e.condition} for all $t \geq T$, then the flow converges exponentially fast to a ball. We recall that we do not need any assumptions on the  regularity, other than $E(0)$ being set of finite perimeter,  or topological assumptions on the initial set. 

\subsection{Application II: Mullins-Sekerka}
The second geometric evolution  equation that we consider, is  the  two-phase Mullins-Sekerka. We will define the flow in the three dimensional flat torus of side length $R$, and assume $R$ large, instead of considering a bounded domain to avoid boundary effects. Note also that the two-phase flow is not well defined in the whole space $\R^3$.  We choose the side length to be large  with respect to the volume of the set, to rule out other possible limit sets than union of  balls. We note that the isoperimetric problem for a general volume in the flat torus is a well-known open problem. 
The Mullins-Sekerka is then given by  the dynamics  
\begin{equation}
\label{eq:mullins}
\begin{cases} 
V_t = [\pa_\nu u_{E(t)}] \quad \text{on } \, \pa E(t)\\
\Delta u_{E(t)} = 0 \quad \text{in } \,  \T_R^3 \setminus \pa E(t)\\
u_{E(t)} = H_{E(t)} \quad \text{on } \, \pa E(t),
\end{cases}
\end{equation}
where $[\pa_\nu u_{E(t)}]$ denotes the jump of the normal. 

Similar to \eqref{eq:VMCF}, also the Mullins-Sekerka is a non-local, volume preserving mean curvature flow. It can also be seen as a gradient flow of the surface area but this time with respect to a suitable $H^{-\frac12}$-Riemannian structure. It can be seen as a quasistatic variant of the Stefan problem  \cite{Luck} and as a singular limit of the Cahn-Hilliard equation, see \cite{AlBaCh, Pego}. For a more comprehensive introduction we refer to \cite{JuMoPoSpa} and the reference therein. 

There is no result similar to \eqref{asym} for the Mullins-Sekerka in the literature. Therefore we need first to prove similar qualitative convergence for \eqref{eq:mullins}, and after that we may prove the exponential convergence similar to Theorem \ref{thm2:MCF}. Here is our result for the asymptotical behavior of \eqref{eq:mullins}. 
\begin{theorem}\label{thm3:mullins}
Fix $v, M>0$. Then there exists $R_0 \geq 1$, depending on $v$ and $M$, with the following property: Let  $\{E(t)\}_{t\geq 0}$  be a  flat flow for \eqref{eq:mullins}  in $\T_R^3$, $R \geq R_0$,  starting from a set of finite perimeter $E(0)\subset\T_R^3$ with volume $|E(0)| = v$ and $P(E(0)) \leq M$. There exist $N \in \N$ and time-dependent points $x_1(t), \dots, x_N(t)$ in $\T_R^3$ such that setting $r = \left( \frac{3 v}{4 \pi N}\right)^{1/3}$ and $F(t) = \bigcup_{i =1}^N B_r(x_i(t))$ it holds $|x_i(t) -x_j(t)| \geq 2r$, $i \neq j$,
\[
| E(t) \Delta F(t)| \to 0  \quad \text{as } \, t \to \infty.
\]
Assume in addition that there exists $\delta_1>0$ and $T_0$ such that 
\[
|x_i(t) -x_j(t)| \geq 2r + \delta_1 \quad \text{ for all } \, t \geq T_0
\]
for $i \neq j$. Then, the sets  $E(t)$ converge exponentially fast to a (independent of time) union of balls
\[
F = \bigcup_{i =1}^N B_r(x_i),\quad \text{ where }\;|x_i -x_j| >2r\quad i \neq j.
\]
Precisely, there exists a constant $C>1$ such that
\[
|E(t) \Delta F| + |P(E(t)) -  N 4 \pi r^2 | \leq C e^{-\frac{t}{C}}.
\]  
\end{theorem}
As in the two-dimensional case, in addition to  Theorem \ref{t:LS} we need also the density estimates due to Sch\"aztle \cite{Sch}  in order to prove a quantitative Alexandrov theorem in terms of the potential $u_{E(t)}$, see Proposition \ref{prop:quantialex-ms}. We remark that we expect  the convergence in Theorem \ref{thm3:mullins} to hold also with respect to the Hausdorff distance, as is the case in Theorem~\ref{thm2:MCF}.

\section{Notation and Preliminaries}

We begin with the notation related to the Euclidian space. We denote the open ball with radius $r$ centered at $x$ by $B_r(x)$ and by $B_r$ if it is centered at the origin.  For $X \in C^1(U;\R^n)$, $U \subset \R^n$ open, we denote its differential by $DX$ and the divergence by $\div X = \text{Tr}(DX)$. Similarly for a real valued function $u \in C^2(U)$ we denote the gradient by $Du$, the Hessian by $D^2 u$ and the Laplacian by $\Delta u$. We point out  that  we will use the symbol $\nabla$  for the covariant derivative on a manifold, which we will introduce shortly. However, with a slight abuse of notation we will still use the symbols $\Delta$ and $''\div''$ for the Laplace-Beltrami operator and the divergence on a manifold respectively. The use will be clear from the context.  

 Given a set $E \subset \R^n$, we define the distance function $\dist(\cdot, E):\R^n\to [0,+\infty)$, as usual $\dist(x, E)= \inf_{y\in E}|x-y|$, and the signed distance function $d_E:\R^n\to \R$ as
\begin{equation}\label{def:signdist}
d_E(x) = \dist(x, E) -  \dist(x, \R^n \setminus E)\,.
\end{equation}
Then clearly it holds $\dist(\cdot, \partial E) =|d_E|$. For any Lebesgue measurable set $ E \subset \R^n$ we denote by $|E|$ its Lebesgue measure and define the perimeter of $E$ in an open set $U\subset \R^n$ as
\begin{equation}
P(E,U)=\sup\Big\{\int_{E}\div X \,dx\,:\, X\in C_0^1(U,\R^n)\,\,\,\text{with}\,\,\,\Vert X\Vert_{L^{\infty}}\leq 1\Big\}\,.
\end{equation}
We write $P(E)= P(E,\R^n)$ and if $P(E)<+\infty$, we say that $E$ is a set of finite perimeter. 
In this case, the reduced boundary of $E$ is denoted by $\partial^*E$, and the unit outer normal to $E$ by $\nu_E$.
Then, $P(E, U) = \mathcal{H}^{n-1}(\partial^* E\cap U)$ for any open set $U$.  For a  given vector field $X \in C^1(\R^n;\R^n)$  we define its tangential differential on $\partial E$ by  $\nabla_{T} X  = D X - (D X   \nu_E) \otimes  \nu_E$ and tangential divergence as  $\div^T X = \text{Tr}(\nabla_{T} X )$. If $E$ is $C^2$-regular, i.e., $\pa E $ is $C^2$-hypersurface, we define the mean curvature $\mathrm{H}_E$ as the sum of the principal curvatures and denote the second fundamental form by $\mathrm{A}_E$. We use the orientation such that $\mathrm{H}_E$ is nonnegative for convex sets. The divergence theorem extends to vector fields 
$X\in C^1(\R^n,\R^n)$ as
\[
 \int_{\partial E} \div^T X\,d\mathcal{H}^{n-1}= \int_{\partial E} \mathrm{H}_E (X\cdot \nu_E)\,d\mathcal{H}^{n-1},
\]
where $\cdot$ is the inner product in $\R^n$.  Recall that $\mathrm{\bar H}_E$ denotes  the integral average of the mean curvature, i.e. $(\int_{\pa E}\mathrm{H}_E\,d\mathcal{H}^{n-1})/ P(E).$
This notation is related to sets in  $\R^n$. We will also need notation related to more general Riemann manifolds and introduce the notation related to this. 

\subsection{Smooth immersions}\label{smoothimmersion}
Let $\Sigma$ be a $2$-dimensional, connected (unless otherwise specified), smooth manifold and consider a smooth immersion $\vec{f}: \Sigma\to \R^3$. 
In this work, we will call $\vec{f}(\Sigma)$ immersed surface and ``smooth'' always  means of class $C^\infty$. 
In such a setting, $\Sigma$ is naturally endowed with a Riemannian metric, which is $g=\vec{f}^*g_{\R^3}$, given in local coordinates as $g_{ij}=\partial_{x^{i}}\vec{f}\cdot \partial_{x^{j}} \vec{f}$. The metric $g$ can then be extended to tensors via the formula
$$g(T,S)=g_{i_{1}s_{1}} \dots g_{i_{k}s_{k}} g^{j_{1}z_{1}}\dots g^{j_{l}z_{l}} T^{ i_1...i_k}_{j_1...j_l} S^{s_1...s_k}_{z_1...z_l },$$
where $(g^{ij})$ is the inverse matrix $(g_{ij})$. We note that in the above and throughout the paper we adopt the Einstein summation convention.  The norm of a tensor $T$ is  then defined as $|T|=\sqrt{g(T,T)}$ and it satisfies the following useful properties
\begin{equation}\label{ffeq23}
|g(S,T)|\leq |S|\,|T|\quad\quad\quad|S+T|\leq |S|+|T|\quad\quad\quad |T* S|\leq C|T|\,|S|.
\end{equation}
Here $T* S$ is a tensor formed by a linear combination of new tensors, each obtained by contracting some indices of the tensors $T$ and $S$ with the metric and/or its inverse, and the constant $C$ depends only on the algebraic “structure” of $T*S$. 

One may define uniquely the  area (or canonical) measure  $\sigma$ on $\Sigma$ by imposing in any chart $(U,\varphi)$ that $\sigma (B)= \int_{\varphi(B)}\sqrt{\det g_{ij}}\circ \varphi^{-1}\,dx$. The measure $\sigma$ is then a complete, regular  Radon measure.

We denote by $\nabla$ the Levi-Civita connection of $(\Sigma,g)$. 
We may extend $\nabla$  uniquely to every bundle of tensors (by defining it in a natural way on $C^{\infty}(\Sigma)$ and by imposing the Leibniz rule and the commutativity with any contraction). In local coordinates it is given by
\begin{align*}
\bigl(\nabla_{X}T\bigr)^{i_1\dots i_r}_{j_1\dots j_s}=X^{k}\biggl[\partial_{x^k}T^{i_1\dots i_r}_{j_1\dots j_s}-\sum_{p=1}^{s}\Gamma^{l}_{kj_{p}}T^{i_1\dots i_r}_{j_1\dots j_{p-1} l\,  j_{p+1}\dots j_s }+\sum_{q=1}^{r}\Gamma^{i_{q}}_{kl}T^{i_1\dots i_{q-1}l\,  i_{q+1}\dots i_r}_{j_1\dots j_s}\biggr],
\end{align*}
where Christoffel symbols $\Gamma_{ij}^k$ are expressed in terms of the coefficients of the metric $g$ as
\begin{equation*}
\Gamma_{ij}^k=\frac{g^{kl}}{2}\left( \partial_{x^i} g_{lj} + \partial_{x^j}g_{il}- \partial_{x^l}g_{ij}\right).
\end{equation*}
We will write $\nabla^mT$ for the $m$-th iterated covariant derivative of $T$ and the formula for the interchange of covariant derivatives, which involves the Riemann tensor, is
\begin{equation} \label{eq:interchange-covariant}
\nabla^{2}T_{abj_1\dots j_s}^{i_1\dots i_r}-\nabla^{2}T_{baj_1\dots j_s}^{i_1\dots i_r}=\sum_{p=1}^{s}R_{abj_pm}g^{ml}T^{i_1\dots i_r}_{j_1\dots j_{p-1}l\,  j_{p+1}\dots j_s }+\sum_{q=1}^{r}R_{abml}g^{mi_q}T^{i_1\dots i_{q-1}l\,  i_{q+1}\dots i_r}_{j_1\dots j_s},
\end{equation}
where we recall 
\begin{equation*}
R_{ijkl}=\Bigl(\partial_{x^j}\Gamma_{ik}^m-\partial_{x^i}\Gamma_{jk}^m+\Gamma_{ik}^s\Gamma_{js}^m-\Gamma_{jk}^s\Gamma_{is}^m\Bigr)g_{ml}.
\end{equation*} 
Finally, the (rough) Laplacian $\Delta T$ of a tensor $T$ is defined as $g^{ij}\nabla^2T_{ij}$.
\smallskip

Let $\vec{f}: \Sigma\to \R^3$ be a smooth immersion of a generic connected smooth $2$-manifold $\Sigma$.
Since $d\vec{f}_p:T_p\Sigma\to \R^3$ is an injective linear map, we identify $T_p\Sigma$ with the linear subspace $d\vec{f}_p(T_p\Sigma)$ of $\R^3$, for all $p\in\Sigma$, and by virtue of this identification, at every point $p \in\Sigma$ we can define up to a sign a unit normal vector $\nu(p)$. 
Let us observe that if $\Sigma$ is closed and $\vec{f}$ is also injective, then the unit normal vector $\nu(p)$ can be chosen so that it depends globally in a $C^\infty$-way on the point $p$. Indeed, if $\Sigma$ is closed and $\vec{f}$ is also injective, then $\vec{f}$ is an embedding and $\vec{f}(\Sigma)$ admits a unique smooth structure making it into an embedded surface of $\R^3$ with the property that $\vec{f}$ is a diffeomorphism onto its image. 
Hence, $\vec{\iota}:\vec{f}(\Sigma)\to \R^3$ is a smooth embedding, the cited smooth structure on $\vec{f}(\Sigma)$ is determined by the smooth atlas formed by the charts $( \vec{f}(U), \varphi\circ {\vec{f}}^{\,\,-1})$, where $(U,\varphi)$ is any chart of $\Sigma$, and in this case, the tangent spaces $T_{\vec{f}(p)} \vec{f}(\Sigma)$ and $T_p\Sigma$ are identified with the same linear subspace of $\R^3$.
Then, by the theorem of Jordan-Brouwer \cite[Proposition 12.2]{Ben}, $\R^3\setminus \vec{f}(\Sigma)$ has two connected components and one of them is a compact smooth $3$-dimensional submanifold of $\R^3$ with boundary $\vec{f}(\Sigma)$. Accordingly, there exists $\bar\nu:\vec{f}(\Sigma)\to \R^3$ global unit normal smooth vector field {\em pointing outward}, as well as $\bar N:\vec{f}(\Sigma)\to \R^3$ global unit normal smooth vector field {\em pointing inward}, $\bar N=-\bar \nu$. Thus, $\nu=\bar\nu \circ \vec{f}$ and $N=\bar N\circ \vec{f}$ will be the outward-pointing  and inward-pointing global unit normal smooth vector field of $\Sigma$, respectively. Therefore, $\Sigma$ is orientable.
In fact, the theorem of Jordan-Brouwer and relative consequences are in general true for every closed, connected, embedded smooth surface $\bar \Sigma$ of $\R^3$. The surface $\bar \Sigma$ is orientable, since {\em an orientation of $\bar \Sigma$, which will be considered from now on, is determined by the charts $(\bar U,\bar \varphi=(x^1,x^2))$ of $\bar \Sigma$ for which
\begin{equation}
\bar N_p\,=\,\frac{\partial_{x^{1}}\times \partial_{x^{2}}}{|\partial_{x^{1}}\times \partial_{x^{2}}|}(p)
\,=\frac{\partial_{x^{1}}\bar\psi\times \partial_{x^{2}}\bar\psi}{\big[\,|\partial_{x^{1}}\bar\psi|^2\,|\partial_{x^{2}}\bar\psi|^2 \,-\,(\partial_{x^{1}}\bar{\psi}\cdot \partial_{x^{2}}\bar\psi)^2\,\big]^{1/2}}\,(\bar\varphi(p))\label{feqnormal}
\end{equation}
for every $p\in \bar U$, where $\bar\psi=\bar\varphi^{-1}$ and $\bar N:\bar\Sigma\to \R^3$ is global unit normal smooth vector field pointing inward,} by noting that the Jacobian determinant of the transition maps between of two charts of $\bar\Sigma$, which satisfy the identity \eqref{feqnormal}, is positive.
In the case $\bar \Sigma=\vec{f}(\Sigma)$, then we obtain the existence of an atlas $\mathcal{A}=\{\big(U_i,\varphi_i=(x^1,x^2)\big)\}_{i\in I}$ of $\Sigma$ such that both
\begin{equation}
N_p=\bar N_{\vec{f}(p)}\,=\,\frac{\partial_{x^{1}} ( \vec{f}\circ \varphi_i^{-1}) \times \partial_{x^{2}} ( \vec{f}\circ \varphi_i^{-1})}{|\partial_{x^{1}}( \vec{f}\circ \varphi_i^{-1})\times \partial_{x^{2}} ( \vec{f}\circ \varphi_i^{-1})|}\,(\varphi_i(p))\,=\,\frac{\partial_{x^{1}}  \vec{f} \times \partial_{x^{2}}\vec{f}}{|\partial_{x^{1}} \vec{f} \times \partial_{x^{2}} \vec{f}|}\,(p),\label{feqnormal1}
\end{equation}
for all $p\in U_i$ and every $i\in I$, and the Jacobian determinants of all its transition maps are positive, hence $\Sigma$ is orientable and its orientation is that one given by the atlas $\mathcal{A}$.
\smallskip

Let $\vec{f}: \Sigma\to \R^3$ be a smooth immersion of a connected, $2$-dimensional, smooth manifold $\Sigma$. As in the case of sets, we may define the second fundamental form and the mean curvature also in this case. However, we prefer to keep the notation from differential geometry as it is standard in this context.   For a choice of the smooth  unit normal vector field $\nu$ around a point $p\in \Sigma$, the (scalar) second fundamental form at $p$ is the symmetric bilinear form on $T_p\Sigma$ defined as 
\begin{equation}\label{defsecfundform}
\mathrm{A}_p(v,w)= \overline{\nabla}_v \nu \cdot w= N\cdot \overline{\nabla}_v w,
\end{equation} 
for all $v, w \in T_p\Sigma$, where $N=-\nu$ and $\overline{\nabla}$ is the Levi-Civita connection of $(\R^3, g_{\R^3})$. 
Then we have the (scalar) mean curvature at $p$, $\mathrm{H}_p$, which is the trace of $\mathrm{A}_p$ with respect $g_p$ (that is $\mathrm{H}_p=\mathrm{A}_p(e_1,e_1)+\mathrm{A}_p(e_2,e_2)$, for any orthonormal basis $\{e_1,e_2\}$ of $T_p\Sigma$), the vector mean curvature at $p$, $\vec{\mathrm{H}}_p=\mathrm{H}_pN=-\mathrm{H}_p\nu$, the traceless part $\ringg{\rm{A}}_p$ of $\mathrm{A}_p$, given by $\ringg{\mathrm{A}}_p= \mathrm{A}_p - (\mathrm{H}_p/2)g_p$, and the Weingarten operator at $p$, which is the linear map $\mathrm{W}_p:T_p\Sigma\to T_p\Sigma$ defined through $g_p(\mathrm{W}_p (v), w)=\mathrm{A}_p (v,w)$ for all $v,w\in T_p\Sigma$.
The eigenvalues $k_1\leq k_2$ of $\mathrm{W}_p$ are the principal curvatures at $p$, and we have  the following identities: $|\mathrm{A}|^2=k_1^2+k_2^2$; $|\ringg{\mathrm{A}}|^2=(k_2-k_1)^2/2$; $\mathrm{H}=k_1+k_2$; $\mathrm{K}^{G}=k_1k_2$. Here, $\mathrm{K}^{G}$ is the Gaussian curvature, which is equal to the half of the scalar curvature of $(\Sigma,g)$.  
Notice that locally, we may always choose $\nu$ to be smooth.

The Riemann tensor, the Ricci tensor and the scalar curvature can be expressed via the second fundamental form as follows
\begin{align*}
R_{ijkl}&=\mathrm{A}_{ik}\mathrm{A}_{jl}-\mathrm{A}_{il}\mathrm{A}_{jk}\\
\mathrm{Ric}_{ij}&=g^{kl}R_{ikjl}=\mathrm{H}\mathrm{A}_{ij}-\mathrm{A}_{il}g^{lk}\mathrm{A}_{kj}\\
\mathrm{Sc}&=g^{ij}\mathrm{Ric}_{ij}=\mathrm{H}^2-|\mathrm{A}|^2 .
\end{align*}
A consequence of this link between the Riemann tensor and the second fundamental form and of the formula for the interchange of the covariant derivatives \eqref{eq:interchange-covariant} is the useful equality
\begin{equation}\label{ffeq21}
\nabla \Delta \mathrm{H}=\Delta \nabla \mathrm{H}+ \mathrm{A}*\mathrm{A}*\nabla \mathrm{H}.
\end{equation}

The symmetry properties of the covariant derivative of $\mathrm{A}$, given by the Codazzi equations $\nabla \mathrm{A}_{ijk}=\nabla \mathrm{A}_{jik}=\nabla \mathrm{A}_{kij}$, imply 
$\nabla\mathrm{H}=\div\mathrm{A}=2\,\div\ringg{\mathrm{A}}$,
which in turn yield the  inequalities
\begin{equation}\label{ffeq16}
|\mathrm{H}|^2\leq 2\,|\mathrm{A}|^2,\quad\quad\quad |\nabla\mathrm{H}|^2\leq 2\,|\nabla\mathrm{A}|^2\leq 12\,|\nabla\ringg{\mathrm{A}}|^2, \quad\quad\quad |\nabla^2\mathrm{H}|^2\leq 2\,|\nabla^{2}\mathrm{A}|^2\leq 12\,|\nabla^2\ringg{\mathrm{A}}|^2
\end{equation}
and the well-known Simons’ identity
\begin{equation}
\label{eq:simon}
\Delta  \mathrm{A}_{ij} = \nabla^2\mathrm{H}_{ij} + \mathrm{H} \mathrm{A}_{il} g^{ls} \mathrm{A}_{sj} - |\mathrm{A}|^2 \mathrm{A}_{ij}.
\end{equation}
 
\vspace{0.15cm}
The {\em Willmore energy} of a smooth immersion $\vec{f}: \Sigma\to \R^3$, for a generic connected, closed, smooth $2$-manifold $\Sigma$, is given by
\begin{equation}\label{Willenergy}
\mathcal{W}(\vec{f})=\frac{1}{4}\int_{\Sigma} \mathrm{H}^2\,d\sigma\,.
\end{equation}
The most important property of this functional is its invariance under conformal diffeomorphisms of $\R^3$, recalling that a smooth immersion $f: (M, g)\to (N, h)$ between two Riemannian manifolds $(M, g)$ and $(N, h)$ is said to be conformal if $f^* h=e^{2\lambda}g$ for some $\lambda\in C^{\infty}(M)$.
More precisely, $\mathcal{W}(\Phi \circ\vec{f})=\mathcal{W}(\vec{f})$, where $\Phi$ is any combination of a translation of $\R^3$ given by $x \mapsto x+x_0$ for a fixed $x_0\in \R^3$,  a dilation of $\R^3$ defined as $x \mapsto \alpha x$ for some $\alpha>0$,  and a spherical inversion $\R^3$ with center $p \notin \vec{f}(\Sigma)$ and radius $r>0$ described as $x \mapsto r^2 (x-p)/|x-p|^2$.
In the case that $\Phi$ is a spherical inversion with center $p\in \vec{f}(\Sigma)$, we have $\cW(\Phi \circ\vec{f})=\cW(\vec{f})-4\pi\,\sharp\vec{f}^{-1}(p)$, where $\sharp(\cdot)$ is the cardinality of $(\cdot)$,  \cite{BaKu}.
As a consequence, we get the Li--Yau inequality \cite{LiYa82}, that is $\cW(\vec{f}) \ge 4\pi\, \sharp \vec{f}^{-1}(p)$ for all $p\in  \vec{f}(\Sigma)$. In particular if $\vec{f}$ is not an embedding, then $\cW(\vec{f})\ge8\pi$.

The infimum of $\mathcal{W}(\vec{f})$ among all smooth immersions $\vec{f}: \Sigma\to \R^3$, where $\Sigma$ is a connected, closed, smooth $2$-manifold, belongs to $[4\pi,8 \pi)$ if all $2$-manifolds $\Sigma$ are orientable. 
More precisely, for each nonnegative integer $\mathfrak{g}$, setting
	\begin{align*}
	\beta_\mathfrak{g}:=\inf \Big\{ \cW(\vec{f})\,\,;\,&\, \vec{f}:\Sigma\to\R^3 \mbox{ is a smooth immersion with $\Sigma$ a $2$-dimensional, } \\
	&\,\mbox{connected, closed, orientable, smooth manifold of genus $\mathfrak{g}$}\Big\},
	\end{align*}
it turns out $4\pi\le\beta_\mathfrak{g}<8\pi$ and each infimum $\beta_\mathfrak{g}$ is attained by a smooth embedding $\vec{F}:\Sigma\to\R^3$, \cite{Wi65, La70, Sim93, BaKu}.
The equality $\beta_\mathfrak{g}=4\pi$ holds if and only if $\mathfrak{g}=0$ and $\vec{F}$ is a round sphere, that is, $\Sigma$ is topologically a sphere and $\vec{F}$ embeds $\Sigma$ as a round $2$-sphere of $\R^3$.
Similarly, $\beta_\mathfrak{g}\geq 2\pi^2$ for any $\mathfrak{g}\geq 1$ and the equality holds if and only if $\mathfrak{g}=1$ and $(\pi^{-1}\circ\vec{F})(\Sigma)$ is the Clifford torus $\SSS^1_{\frac{1}{\sqrt{2}}}\times \SSS^1_{\frac{1}{\sqrt{2}}}$, up to conformal diffeomorphisms of $\SSS^3$, \cite{MN14}, where $\SSS^n_r$ denotes the sphere of radius $r$ and center the origin of $\R^{n+1}$ and $\pi$ is a generic stereographic projection of $\SSS^3$. For completeness, we mention that there exists the limit of the $\beta_\mathfrak{g}$'s for $\mathfrak{g}\to +\infty$ and it holds that $\lim_{\mathfrak{g}\to +\infty}\beta_\mathfrak{g}=8\pi$, \cite{KLS10}.

\subsection{Embedded surfaces and  Riemann surfaces}\label{Riemsurf}

Given a closed connected embedded smooth surface $\Sigma \subset \R^{3}$, the inclusion $\vec{\iota}:\Sigma\to\R^3$ is a smooth immersion. Thus, the above is applicable, recovering in this way the classic theory of the embedded surfaces of $\R^3$. In particular, $\Sigma$ is a Riemannian surface where the metric $g$ is the (natural) one induced by the Euclidian metric, it is orientable with orientation fixed through \eqref{feqnormal}, and the second fundamental form $\mathrm{A}$, the mean curvature $\mathrm{H}$, the traceless part $\ringg{\mathrm{A}}$ of $\mathrm{A}$ can be globally defined in a smooth way on it. In the case that the surface is the boundary of a set $E$, i.e. $\Sigma = \pa E$, we use the notation $\mathrm{H}_E$, $\mathrm{A}_E$, etc., while if we are interested only on the surface, we write $\mathrm{H}_\Sigma, \mathrm{A}_\Sigma$ etc..

Since $\Sigma$ is a closed connected orientable smooth $2$-manifold, it is diffeomorphic either to a sphere or to a connected sum of tori, \cite{BenMan, massey1}. 
If it is diffeomorphic to a sphere, we say that it has genus $0$, while if it is diffeomorphic to the connected sum of $n$ tori, we say that it has genus $n$.
It follows then from Gauss-Bonnet theorem (see for instance \cite{petersen2}) that
\begin{equation}\label{Gauss--Bonnet Theorem}
\int_{\Sigma} \mathrm{K}^G_{\Sigma}\,d\mathcal{H}^2=2\pi\chi(\Sigma)=4\pi\big(1-\text{genus}(\Sigma)\big).
\end{equation}

Another important property for the surface $\Sigma$ is the existence of an atlas $\{\big(U_j, (u_j,v_j)\big)\}_{j\in J}$, compatible with the fixed orientation, such that all local coordinates $(u_j,v_j)$ are isothermal for the Riemannian metric $g$, i.e. $$g=e^{2\lambda_j}(du_j\otimes du_j+dv_j\otimes dv_j),$$ 
with $\lambda_j\in C^{\infty}(U_j)$. Moreover, the family $\{(U_j, u_j+i v_j)\}_{j\in J}$ defines a one-dimensional complex structure on $\Sigma$, since all its transition maps and their inverses are holomorphic, and the Riemann surface $S$ so obtained, that is the one-dimensional complex manifold determined by this complex structure, is said to be {\em induced by the Riemmanian oriented surface} $(\Sigma, g)$.
If $S$ and $S'$ are Riemann surfaces induced by the Riemmanian oriented surfaces $(\Sigma, g)$ and $(\Sigma', g')$ respectively, then $f:(\Sigma, g)\to (\Sigma', g')$ is a smooth orientation-preserving conformal diffeomorphism if and only if $f:S\to S'$ is biholomorhic, that is $f$ is holomorphic with its inverse.
This is still true if $(\Sigma, g)$, as well as $(\Sigma', g')$, is a generic $2$-dimensional oriented Riemannian smooth manifold, therefore, conformal Riemannian metrics determine the same complex structure in $2$-dimensional oriented smooth manifolds, where we recall that two Riemannian metrics $h,h'$ on a same smooth manifold $M$ are conformal if there exists $\lambda\in C^\infty(M)$ such that $h'=e^{2\lambda}h$. More precisely, in the case of $2$-dimensional oriented smooth manifolds, the concepts of structure complex and of conformal structure (that is a conformal equivalence class of Riemannian metrics) are equivalent. 

Above we introduced the notion of Riemann surface and that one of Riemann surface induced by a $2$-dimensional Riemmanian oriented smooth manifold because we want to apply a known corollary of the Uniformization Theorem, which provides a classification up to biholomorphic maps of all connected Riemann surfaces, to the Riemann surfaces induced by closed connected embedded smooth surfaces of $\R^{3}$ with genus $0$.
This corollary states that every closed connected Riemann surface of genus zero is biholomorphic equivalent to the Riemann sphere, that is $\SSS^2\subseteq\R^3$ endowed with the complex structure induced by the Riemmanian oriented surface $(\SSS^2, g_{\SSS^2}=\vec{\iota}^{\,\,*}g_{\R^3})$ whose orientation is given once again through \eqref{feqnormal}. Here, the expression ``biholomorphic equivalent'' means that there exists a biholomorphic map between the considered Riemann surfaces. We refer the reader to \cite{ImTan, Jost} for a detailed treatment of these topics.

Arguing as in the proof \cite[Proposition 3.1] {LSch} and using both the above cited corollary and \cite[Theorem 5.1]{Sch}, we obtain the following  useful  result.

\begin{proposition}
\label{prop:conformal}
For every $\delta_1 \in (0,8\pi)$, there exists a positive constant $C =C(\delta_1)$ with the following property. Let $\Sigma \subset \R^3$ 
be a closed connected embedded smooth surface of genus $0$ with $\H^2(\Sigma) = 4 \pi$ and 
$$\|\ringg{\mathrm{A}}_\Sigma\|_{L^2(\Sigma)}^2 < 8\pi - \delta_1.$$
The Riemmanian surface $\Sigma$ is oriented with orientation given through \eqref{feqnormal}.
Then there is a smooth conformal orientation-preserving diffeomorphism $f : \mathbb{S}^2 \to \Sigma$ with pull-back metric $g = f^* g_\Sigma = e^{2u} g_{\mathbb{S}^2}$ having
$$\|u\|_{L^\infty(\mathbb{S}^2)} \leq C. $$
\end{proposition}

In the following we will also need the Michael-Simon inequality, \cite{MiSi}, that states the existence of constant $C>0$ such that 
\begin{equation}\label{MSinequality}
\|w\|_{L^{2}(\Sigma)} \leq C \int_{\Sigma} |\nabla w| + |\mathrm{H}_\Sigma w|\, d \H^2,
\end{equation}
for every $w \in C^1(\Sigma)$ and for all closed connected embedded smooth surfaces $\Sigma$ of $\R^3$.

\subsection{Weak immersions of the sphere $\SSS^2$, convergence results and spheres minimising the Canham-Helfrich energy}\label{Weakimmersionsec}

Let us consider the Riemmanian oriented surface $(\SSS^2, g_{\SSS^2})$ with orientation given through \eqref{feqnormal}.
The default metric on the sphere $\SSS^2$ will always be the canonical one $g_{\SSS^2}$, therefore, if  not mentioned otherwise, the quantities considered on the sphere $\SSS^2$ are all refereed to the standard metric $g_{\SSS^2}$.

A map $\vec{f}\in W^{2,2}(\SSS^2,\R^3)\cap W^{1,\infty}(\SSS^2,\R^3)$ is called a {\em weak immersion} if the following conditions are satisfied.
\begin{itemize}
\item[$(i)$] There exists $C\geq 1$ such that $$C^{-1}g_{\SSS^2}(p)(v,v)\leq g(p)(v,v)\leq C g_{\SSS^2}(p)(v,v),$$
for every $v\in T_p\SSS^2$ and almost every $p\in \SSS^2$, where $g$ is the $L^\infty$-metric on $\SSS^2$ associated with  $\vec{f}$ given almost everywhere by $g(\partial_{x^i},\partial_{x^j})=\partial_{x^i}\vec{f} \cdot\,\partial_{x^j}\vec{f}$ in local coordinates. Similarly to the smooth case, the metric $g$ induces a Radon measure $\sigma_g$ on $\SSS^2$, called area measure. 
\item[$(ii)$] The Gauss map $N$, or $N_{\vec{f}}$ , defined almost everywhere in any positive local chart of $\SSS^2$ through the formula \eqref{feqnormal1}, belongs to $W^{1,2}(\SSS^2,\R^3)$. Observe that $N\in W^{1,2}(\SSS^2,\R^3)$ if  
 $$\int_{\SSS^2}\!|dN|_g^2\, d\sigma_g<+\infty,$$
where $|dN|_g^2=g^{i k}g^{jl} (\partial_{x^{i}}N\cdot \partial_{x^{j}}\vec{f})\,(\partial_{x^{k}}N\cdot \partial_{x^{l}}\vec{f})$ in local coordinates.
\end{itemize}
We denote by $\mathcal{E}$ the class of the weak immersions $\vec{f}:\SSS^2\to \R^3$ of $\SSS^2$ into $\R^3$.

A map $\vec{f}\in W^{1,\infty}(\SSS^2,\R^3)$ is {\em (weakly) conformal} if $g$ defined as before can be written almost everywhere as $e^{2u}g_{\SSS^2}$ for some $u\in L^\infty(\SSS^2)$ called conformal factor. In this case, the equalities $\partial_{x^1}\vec{f}\cdot\partial_{x^2}\vec{f}=0$ and $|\partial_{x^1}\vec{f}|=|\partial_{x^2}\vec{f}|$ hold almost everywhere for every conformal local chart of $\SSS^2$. 
Finally, we recall that for every weak immersion $\vec{f}:\SSS^2\to \R^3$, there always exists a bilipschitz homeomorphism $\phi$ of $\SSS^2$ such that $\vec{f}\circ \phi$ belongs to $W^{2,2}(\SSS^2,\R^3)\cap W^{1,\infty}(\SSS^2,\R^3)$, $\vec{f}\circ \phi$ is weakly conformal and $N_{\vec{f}}\circ \phi=N_{\vec{f}\,\circ\, \phi}$, \cite{Hel, RivLec}.

Given a weak immersion $\vec{f}:\SSS^2\to \R^3$, one may  define the second fundamental form $\mathrm{A}$, the (scalar) mean curvature $\mathrm{H}$, the vector mean curvature $\vec{\mathrm{H}}$ and the Gaussian curvature $\mathrm{K}^G$ almost everywhere in the same way as  with smooth immersions.


\begin{proposition}(\cite[pp. 81-83]{RivLec})\label{conv1}
Let $\vec{f}_1, \vec{f}_2, \dots\in \mathcal{E}$ be a sequence of conformal weak immersions of the $2$-sphere $\SSS^2$ into $\R^3$.
We assume that there exist a radius $R$ and constant $C>1$ such that
\begin{equation*}
\vec{f}_k(\SSS^2)\subset B_R(0)\quad\quad\quad\quad \Vert u_k\Vert_{L^{\infty}(\SSS^2)}\leq C\quad\quad\quad\quad \int_{\SSS^2}  \mathrm{H}_{k}^2\,d\sigma_{k}\leq C
\end{equation*}
for all $k\in\N$, where the $u_k$'s, the $\sigma_k$'s and the $ \mathrm{H}_{k}$'s are the conformal factors, the area measures and the mean curvatures corresponding to the $\vec{f}_{k}$'s, respectively.\\
Then,  by possibly passing  to a sub-sequence, there exists a map $\vec{f}_{\infty}\in W^{2,2}(\SSS^2,\R^3)\cap W^{1,\infty}(\SSS^2,\R^3)$ such that it is weakly conformal and 
\begin{align}
\log(|\nabla  \vec{f}_{k}|)&\overset{\ast}{\rightharpoonup} \log(|\nabla  \vec{f}_{\infty}|) \quad\text{in}\quad  L^{\infty}(\SSS^2) \label{eq:Prop2.2-1}\\
\vec{f}_k&\rightharpoonup \vec{f}_\infty \quad\quad\quad\quad\text{in}\quad W^{2,2}(\SSS^2,\R^3) \label{eq:Prop2.2-2}.
\end{align}
\end{proposition}

We are interested in functionals involving the mean curvature, the area and the (algebraic) volume. Therefore we need to study their behavior under the  convergence given by \eqref{eq:Prop2.2-1} and \eqref{eq:Prop2.2-2}.

\begin{proposition}(\cite[Lemma 5.2]{RivLec} and \cite[Lemma 3.1]{MonSch})\label{resultconv1}
Let $\vec{f}_1, \vec{f}_2, \dots\in \mathcal{E}$ be a sequence of conformal weak immersions of the $2$-sphere $\SSS^2$ into $\R^3$, $\sigma_1,\sigma_2,\dots$ the corresponding area measures on $\SSS^2$, and $N_1, N_2, \dots $ the corresponding Gauss maps. We assume that
$$\sup_{k\in\N} \int_{\SSS^2} |d N_k|_{g_k}^2d\sigma_{g_k}<+\infty$$
and there exists a weakly conformal map $\vec{f}_{\infty}\in W^{2,2}(\SSS^2, \R^3)\cap W^{1,\infty}(\SSS^2, \R^3)$ such that  \eqref{eq:Prop2.2-1} and \eqref{eq:Prop2.2-2}.
Then, $\vec{f}_{\infty}\in \mathcal{E}$, i.e., it is weak immersion  and, by possibly passing to a  sub-sequence, it holds 
\begin{align}
\lim_{k\to+\infty}\int_{\SSS^2}d\sigma_{k} &= \int_{\SSS^2}d\sigma_\infty \label{feq32}\\
\lim_{k\to+\infty}\int_{\SSS^2} N_{k}\cdot \vec{f}_{k}\,d\sigma_{k} &= \int_{\SSS^2}  N_{\infty}\cdot \vec{f}_{\infty}\,d\sigma_\infty\label{feq33}\\
\lim_{k\to+\infty}\int_{\SSS^2} \mathrm{H}_{k}\,d\sigma_{k} &= \int_{\SSS^2}  \mathrm{H}_\infty\,d\sigma_\infty\,\label{feq34}\\
\int_{\SSS^2}\mathrm{H}^2_\infty\,d\sigma_\infty&\leq \liminf_{k\to+\infty}\int_{\SSS^2}  \mathrm{H}_{k}^2\,d\sigma_{k},\label{feq30}\\
\int_{\SSS^2}\vert dN_{\infty} \vert^2_{g_{\infty}}\,d\sigma_{\infty}&\leq \liminf_{k\to+\infty}\int_{\SSS^2}\vert dN_{k}\vert^2_{g_k} \,d\sigma_k.\label{feq31}.
\end{align}
\end{proposition}

Another ingredient for the proof of Theorem \ref{t:LS} is a compactness result for weak immersions of the $2$-sphere.
In order to state it,  we need to introduce some definitions.
For any weak immersion $\vec{f}\in \mathcal{E}$, we define  {\em Canham-Helfrich energy} as
\begin{equation}\label{Can-Herenergy}
\mathcal{H}_{c_0}(\vec{f})=\frac{1}{4}\int_{\SSS^2}(\mathrm{H}-c_0)^2\,d\sigma_{g},
\end{equation}
where $c_0$ is a given real constant. Observe that if $c_0=0$, the Canham–Helfrich energy of $\vec{f}$ coincides with  Willmore energy $\mathcal{W}(\vec{f})$, which is defined by the formula \eqref{Willenergy}, exactly as with smooth immersions. Finally, we define by
\begin{equation}\label{ffeq9}
\mathcal{A}(\vec{f})=\int_{\SSS^2}d\sigma_{g} \quad\quad \text{and}\quad\quad 
\mathcal{V}(\vec{f})=-\frac{1}{3}\int_{\SSS^2}  N \cdot \vec{f}\,d\sigma_{g}
\end{equation}
the area and the (algebraic) volume of $\vec{f}$. In the case that $\vec{f}$ is a smooth embedding and $N$ is the inner unit normal, $\mathcal{V}(\vec{f})$ coincides with the enclosed volume, by the divergence theorem (see \cite{RupSch} for instance).
The previously mentioned  compactness result is the following.

\begin{theorem}\cite{MonSch}\label{CH energy's minimizing}
Let $c_0\in\R$ and let $A_0, V_0$ be positive real numbers which  satisfy the isoperimetric inequality $36\pi V_0^2\leq A_0^3$.
Consider the problem
\begin{equation}
\inf_{\vec{f}\in\mathcal{E}_{A_0, V_0}}\!\!\mathcal{H}_{c_0}(\vec{f}) \label{HerCahprob} \tag{$\star$}
\end{equation} 
where $\mathcal{E}_{A_0, V_0}=\{\vec{f}\in\mathcal{E}\,:\,\mathcal{A}(\vec{f})=A_0\,\,\text{and}\,\,\mathcal{V}(\vec{f})=V_0\}$.
If there exists a minimizing sequence $\{\vec{f}_k\}_{k\in \N}$ of the problem \eqref{HerCahprob} having
\begin{equation}\label{feq45}
\sup_{k\in \N}\mathcal{W}(\vec{f}_k)<8\pi,
\end{equation}
then the infimum is attained by a smooth embedding $\vec{f}:\SSS^2\to \R^3$.
\end{theorem}

One may obtain Theorem  \ref{CH energy's minimizing}  from \cite[Theorem 1.7]{MonSch} since the assumption \eqref{feq45} on the  low Willmore energy prevents the bubbling phenomenon and the presence of branch points. However, for the reader's convenience  we give here a more  direct argument, but using the results from \cite{MoRibub, MonSch,RivLec}, that leads to Theorem \ref{CH energy's minimizing}.
\begin{proof}
First of all, note that without loss of generality, we may assume that each $\vec{f}_k$ is also weakly conformal (otherwise, one may replace $\vec{f}_k$ with $\vec{f}_k\circ \phi_k$ for an appropriate bilipschitz homeomorphism $\phi_k$ of $\SSS^2$, see above). 
Then, for every $k\in \N$, we have by the  Gauss--Bonnet theorem for conformal weak immersion that 
\begin{equation}\label{ffeq3}
\int_{\SSS^2}\vert d N_{k}\vert_{g_{k}}^2 \,d\sigma_{k}\,\leq \,4\mathcal{W}(\vec{f}_k)\,\leq \,32\pi.
\end{equation}
We note that the Gauss--Bonnet theorem for conformal weak immersions is a  consequence of the extension
of Liouville's equation to the conformal weak immersions, \cite[pp. 96]{RivLec}. Moreover, arguing as in  \cite[Lemma 1.1]{Sim93} we have the following inequalities 
\begin{equation}\label{ffeq2}
\sqrt{A_0/(8\pi)}\,\leq\,\sqrt{\mathcal{A}(\vec{f}_k)/\mathcal{W}(\vec{f}_k)}\,\leq\,\, 
\text{diam}(\vec{f}_k(\SSS^2))\,\leq\,C\sqrt{\mathcal{A}(\vec{f}_k)\,\mathcal{W}(\vec{f}_k)}\,\leq\,C\sqrt{8\pi A_0}\,,
\end{equation}
for some constant $C>0$ independent of $k$.  Indeed for every $\vec{f}\in \mathcal{E}$, the push forward measure $\mu=\vec{f}_{*}(\sigma_g)$ defines a $2$-dimensional integral varifold in $\R^3$, given by the pair $(\vec{f}(\SSS^2), \sharp(\vec{f}^{-1}(x)))$, with square integrable generalized mean curvature $\mathrm{H}_{\mu}$, defined almost everywhere as 
$$\mathrm{H}_{\mu}(x)= \frac{1}{\sharp(\vec{f}^{-1}(x))}\sum_{y\in \vec{f}^{-1}(x)}\!\!\mathrm{H}_{\vec{f}}\,(y)\quad \quad\text{if} \,\,\,\sharp(\vec{f}^{-1}(x))>0$$
and $0$ otherwise. Another consequence of this fact is 
\begin{equation}\label{ffeq5}
\mathcal{W}(\vec{f})\geq 4\pi \lim_{\rho\to 0^+} \frac{\mu(B_{\rho}(x))}{\pi \rho^2}
\end{equation}
for every $x \in \vec{f}(\SSS^2)$, which in turn yields $\mathcal{W}(\vec{f})\geq 4\pi$. We will refer to the above limit as density of $\vec{f}$ at $x$. (See \cite[Section 2.2]{KuLi} for this last part.)

Applying then \cite[Lemma 4.1]{MoRibub} with a diagonal argument, we obtain, by possibly passing to a sub-sequence, that  for every $k$ there exists a smooth conformal orientation-preserving diffeomorphism $\phi_k$ of $\SSS^2$ such
that, for the reparametrized immersion  $\vec{F}_k=\vec{f}_k\circ \phi_k\in \mathcal{E}$  and for the new conformal factor $\lambda_k=\frac{1}{2}\log\big(\frac{1}{2}|\nabla \vec{F}_k|^2 \big)$ the following holds:
there exists a finite set of points $\{a_1, \dots, a_N\}\subset \SSS^2$ such that for any compact subset $K$ of $\SSS^2\setminus \{a_1, \dots, a_N\}$ there exists a constant $C_K$ depending on the compact $K$ and on the bounds of the areas and energies of the $\vec{f}_k$'s such that $\Vert \lambda_k\Vert_{L^{\infty}(K)}\leq~C_K.$
Moreover, since the diameters of the images of our immersions are bounded uniformly, see \eqref{ffeq2}, then by possibly  translating the sets we may assume that the $\vec{F}_k(\SSS^2)$'s are contained in the ball $B_R(0)$.
Then, in any compact subset of $\SSS^2\setminus \{a_1, \dots, a_N\}$ the assumptions of Proposition \ref{conv1} are satisfied, therefore using the  argument from \cite[pp. 81-83]{RivLec}   together with a standard  diagonal argument we deduce that  there exists a map $\vec{F}_\infty:\SSS^2\setminus \{a_1, \dots, a_N\}\to\R^3$ such that $\vec{F}_\infty\in W^{2,2}_{loc}(\SSS^2\setminus\{a_1, \dots, a_N\},\R^3)\cap W^{1,\infty}_{loc}(\SSS^2\setminus\{a_1, \dots, a_N\},\R^3)$, $\vec{F}_\infty$ is weakly conformal,
\begin{align*}
\log(|\nabla  \vec{F}_k|)&\overset{\ast}{\rightharpoonup} \log(|\nabla  \vec{F}_\infty|) \quad\text{in}\quad  L^{\infty}_{loc}(\SSS^2\setminus \{a_1, \dots, a_N\})\\
\vec{F}_k&\rightharpoonup \vec{F}_\infty \quad\quad\quad\quad\,\,\,\text{in}\quad W^{2,2}_{loc}(\SSS^2\setminus\{a_1, \dots, a_N\},\R^3). 
\end{align*}
Moreover, the equalities and inequalities of Proposition \ref{resultconv1}, i.e. \eqref{feq32}--\eqref{feq31}, are true replacing $\SSS^2$ with $\SSS^2\setminus \cup_{j=1}^N B_{\varepsilon}(a_j)$, for any $\varepsilon>0$.
Joining all this information with the fact that the $\vec{F}_k$'s have the same area and are contained in the same ball, one may  also deduce that $\vec{F}_\infty$ extends to a (weakly conformal) map in $W^{1,2}(\SSS^2,\R^3)$,
as well as $N_{\infty}$ and $e^{2\lambda_\infty}\mathrm{H}^2_{\infty}$ extend to a map in $W^{1,2}(\SSS^2,\R^3)$ and in $L^{2}(\SSS^2)$ respectively, 
and that 
$$\vec{F}_k\overset{\ast}{\rightharpoonup} \vec{F}_\infty \quad\text{in}\quad  L^{\infty}(\SSS^2)\quad\quad\text{and} \quad \quad \vec{F}_k\rightharpoonup \vec{F}_\infty\quad\text{in}\quad W^{1,2}(\SSS^2,\R^3),$$
see \cite[pp. 89-90]{RivLec}. From \cite[Lemma 6.2]{RivLec} and \cite[Section 4.2.3]{KeMoRi} it follows then that $\vec{F}_\infty$ belongs to $W^{1,\infty}(\SSS^2,\R^3)$, which in turn implies that $\vec{F}_\infty\in W^{2,2}(\SSS^2,\R^3)$, and that for each $a_j$ there exists a positive integer $n_j$ such that the singularity can be removed if $n_j=1$, 
and the density of $\vec{F}_\infty$ at $\vec{F}_\infty(a_i)$ is always bigger or equal to $n_j$. Therefore, being 
 $$4\pi n_j\leq \mathcal{W}(\vec{F}_\infty)\leq\liminf_{k\to+\infty}\mathcal{W}(\vec{F}_k)<8\pi,$$ 
one obtains that $\vec{F}_\infty\in \mathcal{E}$.

It remains to show that
\begin{equation}\label{ffeq4}
\mathcal{H}_{c_0}(\vec{F}_\infty)\leq \inf_{ \vec{f}\in\mathcal{E}_{A_0, V_0}}\mathcal{H}_{c_0}(\vec{f})\,\quad\quad \quad\quad \mathcal{A}(\vec{F}_\infty)=A_0\quad\quad\quad \quad \mathcal{V}(\vec{F}_\infty)=V_0
\end{equation}
in order to deduce that  $\vec{F}_\infty$ is a minimizer of the problem \eqref{HerCahprob}. In particular, $\vec{F}_\infty$ then  satisfies Euler-Lagrange equation  associated with the Canham-Helfrich energy in a weak sense and, as a consequence, it is smooth by following the proof of \cite[Theorem 4.3]{MonSch}.

Notice that in order to have \eqref{ffeq4}, it is enough to show 
\[
\lim_{\varepsilon\to 0^+}\liminf_{k\to+\infty} \mathcal{A}\big(\vec{F}_k\,|_{B_{\varepsilon}(a_j)}\,\big)\to 0,
\]
for each $a_j$. Again this is a consequence of the fact that $\sup_{k\in\N}\mathcal{W}(\vec{F}_k)<8\pi$. Indeed,  we argue by contradiction and assume  there exists $j_0\in\{1,\dots,N\}$ for which 
$\lim \limits_{\varepsilon\to 0^+}\liminf \limits_{k\to+\infty}\mathcal{A}\big(\vec{F}_k\,|_{B_{\varepsilon}(a_{j_0})}\big)>0$. Then by \cite[Lemma 4.4]{KeMoRi} $\lim \limits_{\varepsilon\to 0^+}\liminf\limits_{k\to+\infty}\mathcal{W}\big(\vec{F}_k\,|_{B_{\varepsilon}(a_{j_0})}\big)\geq 4\pi$. However, this is impossible as 
\begin{align*}
8\pi>
\lim_{\varepsilon\to 0^+}\liminf_{k\to+\infty}\Big(\mathcal{W}\big(\vec{F}_k\,|_{\cup_{j=1}^N B_{\varepsilon}(a_{j})}\,\big)+\mathcal{W}\big(\vec{F}_k\,|_{\SSS^2\setminus\cup_{j=1}^N B_{\varepsilon}(a_{j})}\,\big)\Big)\geq 4\pi +\mathcal{W}(\vec{F}_\infty)\geq 8\pi\,.
\end{align*}

Finally, since  $\vec{F}_\infty$ is a smooth conformal map, it is, in particular, a smooth immersion. The fact that $\vec{F}_\infty$ is injective, hence a smooth embedding, follows from \eqref{ffeq5} since 
the density of $\vec{F}_\infty$ at $x$, for any $x\in\vec{F}_\infty(\SSS^2)$, coincides with $\sharp(\vec{F}_{\infty}^{-1}(x))$, and $\mathcal{W}(\vec{F}_\infty)<8\pi$.
\end{proof}

\section{Consequences of the non-sharp quantitative Alexandrov theorem \cite{JN}}


Let us begin by restating \cite[Theorem 1.2]{JN} in a slightly different way. 
\begin{theorem}
\label{thm:JN}
Let $E \subset \R^3$ be $C^2$-regular set such that  $P(E) \leq C_0$ and $|E| = |B_1|$. There are constants $q \in (0,1], C>1$ and $\tau >0$, depending only on $C_0$, such that if $\|\mathrm{H}_E-\mathrm{\bar H}_E\big\|_{L^2(\partial E)}\leq \tau $ then $C^{-1} \leq \mathrm{\bar H}_E \leq C$ and there are points $x_1, \dots, x_N$ with $|x_i - x_j|\geq 2 \rho$ for $i \neq j$, where $\rho = 2 \mathrm{\bar H}^{-1}$, such that  for $F = \cup_{i=1}^N B_\rho(x_i)$ it holds 
\[
\sup_{x\in E \Delta F} |d_{F}| + \big|P(E) - 4 \pi N^{\frac13} \big| \leq C \big\|\mathrm{H}_E-\mathrm{\bar H}_E\big\|_{L^2(\partial E)}^q.
\]
\end{theorem}

Indeed, Theorem \ref{thm:JN} follows from \cite[Theorem 1.2]{JN}, since the volume constraint $|E|=|B_1|$ implies 
\begin{equation}\label{ffeq1}
|N- \rho^{-3}|\leq C(C_0)\|\mathrm{H}_E-\mathrm{\bar H}_E\|_{L^2(\de E)}^q.
\end{equation}

Theorem \ref{thm:JN} implies the following non-sharp version of Theorem \ref{t:LS}.

\begin{corollary}\label{l.vesa}
There exists $\tau_0\in (0,1)$ such that for every $C^2$-regular set $E \subset \R^3$ which satisfies  \eqref{e.condition} 
and $\big\|\mathrm{H}_E-\mathrm{\bar H}_E\big\|_{L^2(\partial E)}\leq \tau_0$
it holds 
\begin{equation}\label{e.vesa}
|\mathrm{\bar H}_E-2| \,+ \sup_{x\in E \Delta B_1(x_0)} \big||x - x_0|-1\big| + \big|P(E) - P(B_1)\big| \leq C \big\|\mathrm{H}_E-\mathrm{\bar H}_E\big\|_{L^2(\partial E)}^q,
\end{equation}
where $q\in (0,1]$ and $C>0$ are constants depending only on $\delta_0$ and $x_0 \in \R^3$.
\end{corollary}

\begin{proof}
Theorem \ref{thm:JN} and  the condition \eqref{e.condition} yield
$$
4 \pi N^{\frac13} \leq P(E) + C\tau_0^q \leq 4 \pi \sqrt[3]{2}+ C\tau_0^q -\delta_0,
$$
for a number $N \in \N$. When $\tau_0$ is small enough the above implies $N < 2$ and thus $N= 1$. The inequality \eqref{ffeq1} implies  the rest of the statement.
\end{proof}

In turn, Corollary \ref{l.vesa} allows us to establish the following powerful result, which  contains  information on the topology of the sets.

\begin{corollary}\label{l.fra1}
There exists $\tau_0\in (0,1)$ such that, for every $C^\infty$-regular set $E \subset \R^3$ which satisfies  \eqref{e.condition} 
and $\big\|\mathrm{H}_E-\mathrm{\bar H}_E\big\|_{L^2(\partial E)}\leq \tau_0$, the boundary $\partial E$ of $E$ is a closed connected embedded smooth surface of $\R^3$ with genus $0$, which admits a smooth conformal orientation-preserving diffeomorphism 
$f : \mathbb{S}^2 \to \partial E$ and the conformal factor  $u\in C^{\infty}(\SSS^2)$,  defined by $f^* g_{\pa E} = e^{2u} g_{\mathbb{S}^2}$, satisfies
\begin{equation}\label{feq27}
\|u\|_{L^\infty(\mathbb{S}^2)} \leq C,
\end{equation}
where $C$ is a positive constant depending only on $\delta_0$. Here, the fixed orientation of $\pa E$ is that one given through \eqref{feqnormal}.
\end{corollary}
Observe that $E$ is connected as a consequence of the connectedness of $\de E$.
\begin{proof}
By Corollary \ref{l.vesa}, we have for $\rho=2\mathrm{H}_E^{-1}$  that
\begin{align}
\int_{\de E} \mathrm{H}_E^2 \, d\H^2 - 16 \pi 
&=\int_{\de E} \mathrm{H}_E^2 \, d\H^2 - \mathrm{\bar H}_E^2\, P(E) + \mathrm{\bar H}_E^2\,\big( P(E)-P(B_1)\big)+4\pi\,\mathrm{\bar H}_E^2\,\big(1-\rho^2\big) \nonumber\\
&\leq C\, \big\|\mathrm{H}_E-\mathrm{\bar H}_E\big\|_{L^2(\partial E)}^q\leq C\,\tau_0^q, \label{feq9}
\end{align}
where  $C>0$ is a constant depending only on $\delta_0$. Choosing then $\tau_0>0$ small enough, the boundary $\pa E$ of $E$ is connected and of genus zero due to the properties of the Willmore functional discussed at the end of  Subsection \ref{smoothimmersion}.
Now we use Gauss-Bonnet theorem \eqref{Gauss--Bonnet Theorem}, the fact that the genus of $\pa E$ is zero  and the estimate \eqref{feq9} in order to write 
\begin{align}
\int_{\de E} |\ringg{\mathrm{A}}_{E}|^2\, d\mathcal{H}^2&=\frac{1}{2}\int_{\de E} \mathrm{H}_{E}^2\, d\mathcal{H}^2-2\int_{\de E} \mathrm{K}^G_{E}\, d\mathcal{H}^2 
=\frac{1}{2}\int_{\de E} \mathrm{H}_{E}^2\, d\mathcal{H}^2- 4\pi \chi(\de E)\nonumber\\
&\leq C\, \big\|\mathrm{H}_E-\mathrm{\bar H}_E\big\|_{L^2(\partial E)}^q
\leq C\tau_0^{q}, \label{ffeq6}
\end{align}
where $C>0$ is always a constant depending only on $\delta_0$. Recall that $\ringg{\mathrm{A}}_{E}$ denotes the traceless part of the second fundamental form and $\mathrm{K}^G_{E}$ the Gaussian curvature of $\pa E$, defined in Subsection \ref{smoothimmersion}.
Therefore, by decreasing $\tau_0>0$ if necessary, we may assume that $\Vert \ringg{\mathrm{A}}_{E}\Vert_{L^2(\de E)}^2<2\pi$.
This allows us to  apply Proposition \ref{prop:conformal} to the set $F=\alpha E $ such that $P(F)=4\pi$, as the dilatation invariance yields 
$\int_{\de F} |\ringg{\mathrm{A}}_{F}|^2\, d\mathcal{H}^2=\int_{\de E} |\ringg{\mathrm{A}}_{E}|^2\, d\mathcal{H}^2\,$. Hence, 
 there exist a universal constant $\bar C>0$ and a smooth conformal orientation-preserving diffeomorphism $f:\SSS^2 \to \de E$ such that
\begin{equation}\label{feq25}
\Vert u+\log \alpha\Vert_{L^\infty(\SSS^2 )}\leq \bar C,
\end{equation}
where $f^*g_{\de E}=e^{2u} g_{\SSS^2}$. We obtain the estimate \eqref{feq27} from the conditions in \eqref{e.condition}.
\end{proof}

\section{Proof of Theorem \ref{t:LS}}\label{sec:3}

We will prove Theorem \ref{t:LS} with a variational argument by converting the inequality into a minimization problem. The core of the paper is to prove the following result, which then implies Theorem \ref{t:LS} by a rather straightforward argument. 
\begin{proposition}\label{t.main}
There exists $\bar\eps=\bar\eps(\delta_0)\in (0,1)$ such that for all $\eps\in (0, \bar \eps)$ the functional
\begin{equation}
J_\eps(E) = \int_{\pa E} |\mathrm{H}_E- {\mathrm{\bar H}}_E|^2 \, d\H^2 - \eps\, P(E) \label{ourfunct} \tag{$\star\,\star$}
\end{equation}
is minimized by the unit ball $B_1$ in the class of $C^2$-regular sets $E \subset \R^3$ which satisfy \eqref{e.condition}, i.e.,  $|E|=|B_1|$ and $P(E)\leq 4 \pi \sqrt[3]{2}  -\delta_0$.
\end{proposition}

\begin{proof}
Let us divide the proof in two steps.

\textbf{Step 1:} \emph{There exists $\bar\eps\in (0,1)$ such that for all $\eps\in (0, \bar \eps)$, the functional $J_{\varepsilon}$ admits a smooth minimizer $\e=\e(\varepsilon)$ in the class of the considered sets.}
We begin by choosing a minimizing sequence $(E_n)_n$ of $C^2$-regular sets which satisfy \eqref{e.condition} and converges to the infimum value of $J_\eps$. 
By an approximation argument, we may assume that the minimizing sequence  $(E_n)_n$  is smooth, i.e.,  $\pa E_n$  are $C^\infty$-surfaces.
We may also assume small oscillation of the mean curvature:
\begin{equation}\label{e.small osc}
\big\|\mathrm{H}_{E_n}- {\mathrm{\bar H}}_{E_n}\big\|_{L^2(\partial E_n)}^2 \leq 8 \pi  \bar \eps.
\end{equation}
Indeed, we may always assume the minimizing sequence to have smaller energy than the ball, i.e.,  $J_\eps(E_n)\leq J_\eps(B_1)$. Therefore, by the perimeter bound in \eqref{e.condition}, it holds 
\begin{equation*}
 \int_{\pa E_n} |\mathrm{H}_{E_n}- {\mathrm{\bar H}}_{E_n}|^2 \, d\H^2 \leq \eps\left(P(E_n) -  P(B_1)\right)\leq  4 \pi \sqrt[3]{2}  \eps \leq 8 \pi \bar \eps.
\end{equation*}
As important consequences of \eqref{e.small osc} when $\bar \eps$ is small enough, we have by Corollary \ref{l.vesa} and its consequences \eqref{feq9}-\eqref{ffeq6} that 
\begin{align}
B_{1-C\bar \eps^{q/2}}(x_n) \subset  E_n\subset B_{1+C\bar \eps^{q/2}}&(x_n)\label{ffeq31}\\
| \mathrm{\bar H}_{E_n}-2|\leq&\, C \bar \eps^{\frac{q}{2}} \label{ffeq11}\\
0 \leq P(E_n) - P(B_1) \leq& \,C  \bar \eps^{\frac{q}{2}}   \label{eq:step1-2} \\ 
\int_{\de E_n} \mathrm{H}_{E_n}^2 \, d\H^2 - 16 \pi \leq& \,C \bar \eps^{\frac{q}{2}} \label{ffeq7} \\
\int_{\de E_n} |\ringg{\mathrm{A}}_{E_n}|^2\, d\mathcal{H}^2\leq& \,C \bar \eps^{\frac{q}{2}}. \label{ffeq8}
\end{align}
Since the problem is translation invariant, we may  assume that each $x_n$ is the origin.
Furthermore, again for $\bar \eps$ small enough, Corollary~\ref{l.fra1} guarantees the existence of a uniform constant $C$ and a smooth conformal orientation-preserving diffeomorphism $\vec{f}_n : \mathbb{S}^2 \to \pa E_n\subset \R^3$
such that
\begin{equation}
\|u_n\|_{L^\infty(\mathbb{S}^2)} \leq C,
\end{equation}
where $u_n\in C^{\infty}(\SSS^2)$ is given by $ \vec{f}_{n}^{*} g_{\pa E_n}= e^{2u_n} g_{\mathbb{S}^2}$. We recall that as usual, the fixed orientation of $\pa E_n$ is that one given through \eqref{feqnormal}.

At this point, we observe that  by the above results we may write the energy  $J_\eps(E_n)$ using the parametric approach constructed in Subsection \ref{feqnormal} and have 
\begin{equation}\label{ffeq10}
J_\eps(E_n)=\int_{\SSS^2} (\mathrm{H}_{\vec{f_n}}-\mathrm{\bar H}_{\vec{f}_n})^2\,d\sigma_{g_{\vec{f}_n}}\!\!-\varepsilon \mathcal{A}( \vec{f}_n)=:J_\eps(\vec{f}_n)
\end{equation}
and $\mathcal{V}(\vec{f}_n)=|B_1|$, where $\mathcal{A}(\cdot)$ and $\mathcal{V}(\cdot)$ are the area and the (algebraic) volume, defined in~\eqref{ffeq9}.
These equalities follow from the fact that given any positive chart $(U, \varphi)$ of $\SSS^2$, then $(\vec{f}_n(U), \varphi \circ \vec{f}_n^{-1})$ is a positive chart of $\pa E_n$. Indeed, in such local charts, one can check
\begin{equation*}
\begin{matrix}
&\partial_{x^i}^{\pa E_n}|_q=\partial_{x^i}\vec{f}_n|_{\vec{f}_n^{-1}(q)}&g^{\pa E_n}_{ij}(q)= g^{\vec{f}_n}_{ij}(\vec{f}_n^{-1}(q))&\sqrt{\det (g^{\pa E_n}_{ij})(q)}=\sqrt{\det ( g^{\vec{f}_n}_{ij})(\vec{f}_n^{-1}(q))}\,\\
&N^{\pa E_n}(q)=N^{\vec{f}_n}(\vec{f}_n^{-1}(q))\,\, &\mathrm{A}^{\pa E_n}_{ij}(q)=\mathrm{A}^{\vec{f}_n}_{ij}(\vec{f}_n^{-1}(q))\,&\mathrm{H}_{\pa E_n}(q)=\mathrm{H}_{\vec{f}_n}(\vec{f}_n^{-1}(q)).
\end{matrix}
\end{equation*}
Joining all, one gets the equality \eqref{ffeq10}, while the identity $\mathcal{V}(\vec{f}_n)=|B_1|$ follows by also applying the divergence theorem.

Therefore  the sequence of the smooth conformal immersions $\vec{f}_n$ of the $2$-sphere $\SSS^2$ into $\R^3$ satisfy the assumptions of Proposition \ref{conv1} and by  \eqref{ffeq3} also the assumptions of Proposition \ref{resultconv1}.
Thus, by possibly passing to a converging sub-sequence,  we deduce that at the limit there exists a conformal weak immersion $\vec{f}_{\infty}\in \mathcal{E}$ of the $2$-sphere $\SSS^2$ into $\R^3$ such that the equalities and inequalities \eqref{feq32}--\eqref{feq31} are true.

Using the  notation from Section \ref{Weakimmersionsec}, we then have 
\begin{equation*}
\inf_{\substack{E\subset \R^3 \\ \text{satisfies \eqref{e.condition}}}}\!\!\!\!\! J_{\varepsilon}(E)\,=\!\lim_{n\to +\infty} J_{\varepsilon}(E_{n})=\!\lim_{n\to +\infty} J_{\varepsilon}(\vec{f}_{n})\,\geq\, J_{\varepsilon}(\vec{f}_\infty)=\!\int_{\SSS^2} (\mathrm{H}_\infty-\mathrm{\bar H}_\infty)^2\,d\sigma_{\infty} -\varepsilon \mathcal{A}(\vec{f}_\infty)
\end{equation*}
and $\mathcal{V}(\vec{f}_\infty)=|B_1|$. The map $\vec{f}_{\infty}\in \mathcal{E}$ is then the bridge between the two variational problems by virtue of the following crucial chain of inequalities,
\begin{equation}\label{ffeq14}
\inf_{\substack{E\subset \R^3 \\ \text{satisfies \eqref{e.condition}}}}\!\!\!\!\! J_{\varepsilon}(E)\,\geq \,J_{\varepsilon}(\vec{f}_\infty) \,\geq\, 4 \bigg(\inf_{\vec{h}\in\mathcal{E}_{A_\infty,V_\infty }}\!\!\!\mathcal{H}_{\mathrm{\bar H}_\infty}(\vec{h})\,\bigg) -\varepsilon A_\infty,
\end{equation}
where $A_\infty=\mathcal{A}(\vec{f}_\infty)$,  $V_\infty= \mathcal{V}(\vec{f}_\infty)$ and $\mathcal{E}_{A_\infty, V_\infty}=\{\vec{f}\in\mathcal{E}\,:\,\mathcal{A}(\vec{f})=A_\infty\,\,\text{and}\,\,\mathcal{V}(\vec{f})=V_\infty\}$. Here, $\mathcal{H}_{\mathrm{\bar H}_\infty}(\vec{h})$ is the Canham-Helfrich energy of the weak immersion $\vec{h}\in \mathcal{E}$, defined in \eqref{Can-Herenergy} with $c_0 = \mathrm{\bar H}_\infty$. Therefore,  next we show that the assumptions of Proposition \ref{CH energy's minimizing} are satisfied.

We begin by observing that joining the limit information \eqref{feq32}--\eqref{feq30} with the inequalities \eqref{e.small osc}, \eqref{ffeq11}, \eqref{eq:step1-2} we know that the  estimates
\begin{equation}\label{ffeq12}
0\leq A_\infty\!-4\pi\leq C \bar\eps^{\frac{q}{2}}\quad\quad\text{and} \quad\quad| \mathrm{\bar H}_{\infty}-2|\leq\, C \bar \eps^{\frac{q}{2}} 
\end{equation}
hold with the same uniform constant $C>0$ of \eqref{ffeq11}-\eqref{eq:step1-2}, and
\begin{equation}\label{ffeq13}
\int_{\SSS^2}(\mathrm{H}_{\infty}- {\mathrm{\bar H}}_{\infty})^2\,d\sigma_{\infty}\leq 8 \pi  \bar \eps.
\end{equation}
The  first  consequence of \eqref{ffeq12} is that the positive constants $A_\infty, V_\infty$  satisfy  the isoperimetric inequality $36\pi V_\infty^2\leq A_\infty^3$, as $\mathcal{V}_\infty=|B_1|$.  Let $\vec{h}_1, \vec{h}_2, \dots\in \mathcal{E}$ be a minimizing sequence for the problem $\inf_{\vec{h}\in\mathcal{E}_{A_\infty, V_\infty}}\!\!\!\mathcal{H}_{\mathrm{\bar H}_\infty}(\vec{h})$, $\widetilde{\sigma}_1,\widetilde{\sigma}_2,\dots$ the corresponding area measures, and $\widetilde{\mathrm{H}}_1, \widetilde{\mathrm{H}}_2, \dots $ the associated mean curvatures. By Minkowski inequality 
\[
\| \widetilde{\mathrm{H}}_n\|_{L^2(\SSS^2,\,\widetilde{\sigma}_n)}\leq \| \widetilde{\mathrm{H}}_n-\mathrm{\bar H}_\infty \|_{L^2(\SSS^2,\, \widetilde{\sigma}_n)} +\| \mathrm{\bar H}_\infty \|_{L^2(\SSS^2,\, \widetilde{\sigma}_n)} =  \| \widetilde{\mathrm{H}}_n-\mathrm{\bar H}_\infty \|_{L^2(\SSS^2,\, \widetilde{\sigma}_n)} +\mathrm{\bar H}_\infty \sqrt{A}_\infty.
\]
Since $\vec{f}_\infty \in \mathcal{E}_{A_\infty, V_\infty}$, we may assume that $\mathcal{H}_{\mathrm{\bar H}_\infty}(\vec{h}_n)\leq \mathcal{H}_{\mathrm{\bar H}_\infty}(\vec{f}_\infty)$ for all $n\in\N$. Therefore by   \eqref{ffeq12} and \eqref{ffeq13}  it holds 
\begin{equation*}
\| \widetilde{\mathrm{H}}_n\|_{L^2(\SSS^2, \,\widetilde{\sigma}_n)}\leq \| \widetilde{\mathrm{H}}_\infty-\mathrm{\bar H}_\infty \|_{L^2(\SSS^2,\,\sigma_{\infty})}  +\mathrm{\bar H}_\infty \sqrt{A}_\infty
\,\leq\, \sqrt{8 \pi  \bar \eps}+(2+C \bar \eps^{\frac{q}{2}}) \sqrt{4\pi+C \bar\eps^{\frac{q}{2}}}.
\end{equation*}
As a consequence, for $\bar\eps>0$ sufficiently small, the minimizing sequence of the $\vec{h}_n$'s satisfies the condition \eqref{feq45} (recall that $\mathcal{W}(\vec{h}_n)=\|\widetilde{\mathrm{H}}_n\|_{L^2(\SSS^2, \,\widetilde{\sigma}_n)}^2/4$). Thus, applying Theorem \ref{CH energy's minimizing}, 
the infimum of the problem $\inf_{\vec{h}\in\mathcal{E}_{A_\infty, V_\infty}}\!\!\!\mathcal{H}_{\mathrm{\bar H}_\infty}(\vec{h})$ is attained by a smooth embedding $\vec{h}:\SSS^2\to \R^3$.
Therefore, as explained in Subsection \ref{smoothimmersion}, $\vec{h}(\SSS^2)$ is the $C^\infty$-boundary of a smooth bounded set $\e\subset\R^3$, and it also holds  $P(\e)=A_\infty$, $|\e|=|B_1|$ and $$\mathcal{H}_{\mathrm{\bar H}_\infty}(\vec{h})=\frac{1}{4}\int_{\pa \e} (\mathrm{H}_{\e}-\mathrm{\bar H}_\infty)^2\, d\mathcal{H}^2.$$
The conclusion then follows by  observing that
\begin{align*}
\inf_{\substack{E\subset \R^3 \\ \text{satisfies \eqref{e.condition}}}}\!\!\!\!\! J_{\varepsilon}(E)&\,\geq 4 \bigg( \inf_{\vec{h}\in\mathcal{E}_{A_\infty,V_\infty }}\!\!\!\mathcal{H}_{\mathrm{\bar H}_\infty}(\vec{h})\,\bigg) -\varepsilon A_\infty  \\
&\,\geq \, \int_{\pa \e} (\mathrm{H}_{\e}-\mathrm{\bar H}_\infty)^2\, d\mathcal{H}^2-\varepsilon P(\e)\\
&\,\geq\,\int_{\pa \e} (\mathrm{H}_{\e}-\mathrm{\bar H}_{\e})^2\, d\mathcal{H}^2-\varepsilon P(\e)\\
&\,\geq\inf_{\substack{E\subset \R^3 \\ \text{satisfies \eqref{e.condition}}}}\!\!\!\!\!J_{\varepsilon}(E).
\end{align*}
Hence,  $\e$ is a smooth minimizer of the functional $J_{\varepsilon}$ in the class of $C^2$-regular sets $E \subset \R^3$ which satisfy \eqref{e.condition}.

\medskip
\textbf{Step 2:} \emph{For each $\eps\in (0, \bar \eps)$, possibly choosing a smaller $\bar\eps>0$, every smooth minimizer $\e=\e(\varepsilon)$ of the functional $J_{\varepsilon}$ is the unit ball, i.e. up to a translation it holds $\e= B_1$.}  

Since $\e$  is a smooth minimizer of the functional $J_{\varepsilon}$ it satisfies  the associated  Euler-Lagrange equation
\begin{equation}
\label{eq:eulerlagr}
-\Delta \mathrm{H}_{\e} = \vert \mathrm{A}_{\e}\vert^2(\h_{\e}-\bar\h_{\e}) - \frac{1}{2} (\h_{\e}-\bar\h_{\e})^{2\,}\h_{\e} +  \frac{\eps}{2}\,\h_{\e}+ \lambda \qquad \text{on } \, \pa \e 
\end{equation}
in the classical sense, where $\lambda$ is the Lagrange multiplier due to the volume constraint. The derivation of \eqref{eq:eulerlagr} is standard and thus we omit it. Next we prove uniform regularity estimates for $F$ using the priori bounds and the  Euler-Lagrange equation \eqref{eq:eulerlagr}.  In the following $C$ will denote a constant which is independent of $\varepsilon\in (0,\bar \varepsilon)$ and may vary from a line to line.


We begin by noticing that $\e$ also satisfies (for same reason of the $E_n$'s) the  estimates \eqref{e.small osc} and \eqref{ffeq31}-\eqref{ffeq8}. 
In order to estimate the Lagrange multiplier $\lambda$, we integrate \eqref{eq:eulerlagr} over $\pa \e$, use $\int_{\pa \e} \Delta\mathrm{H}_{\e}\,  d\H^2 = 0$ and the Young's inequality. From the estimates on the perimeter and on the integral average of the mean curvature it follows then that 
\begin{equation*}
 |\lambda | P(\e) \leq   C\int_{\pa \e }(1+ |\mathrm{A}_{\e}|^3) \, d\H^2\leq  C\int_{\pa \e }(1+ |\mathrm{A}_{\e}|^4) \, d\H^2,
\end{equation*}
which yields 
\begin{equation}\label{ffeq15}
|\lambda | \leq C\big(1+ \| A_{\e} \|_{L^4(\pa \e)}^4\big),
\end{equation}
as $P(\e)\geq 4\pi$.
We continue by multiplying the Euler-Lagrange equation \eqref{eq:eulerlagr} by $ \mathrm{H}_{\e}$ and integrating by parts. 
After estimating the right hand side similarly to before, we then obtain 
\begin{equation}\label{ffeq17}
\|\nabla \mathrm{H}_{\e} \|_{L^2(\pa \e)}^2 \leq  C(1+ \| \mathrm{A}_{\e} \|_{L^4(\pa \e)}^4) + |\lambda|\,P(\e) \,\bar\h_{\e}\leq C(1+ \| A_{\e} \|_{L^4(\pa \e)}^4 ), 
\end{equation}
where the last inequality follows from \eqref{ffeq11}  and \eqref{ffeq15}. We then recall the Simon's identity \eqref{eq:simon}. We muplity \eqref{eq:simon} by $\mathrm{A}_{\e}$ and integrate by parts, 
\[
\|\nabla \mathrm{A}_{\e} \|_{L^2(\pa \e)}^2 \leq \|\nabla\mathrm{H}_{\e} \|_{L^2(\pa \e)}^2  +   C \|\mathrm{A}_{\e}  \|_{L^4(\pa \e)}^4.
\]
Hence, we have by the two above 
\begin{equation}
\label{eq:step3-2}
\|\nabla \mathrm{A}_{\e} \|_{L^2(\pa \e)}^2  \leq  C(1+\| \mathrm{A}_{\e} \|_{L^4(\pa \e)}^4) \leq C\big( 1+\big \| \mathrm{A}_{\e} -\frac{1}{2}\,\mathrm{\bar H}_{\e}\, g \big\|_{L^4(\pa \e)}^4\big).
\end{equation}
Notice that in these last two passages we used  classical properties on the norm of tensors, see \eqref{ffeq23}.
We proceed by using the Michael-Simon inequality \eqref{MSinequality} with the function $u= | \mathrm{A}_{\e} -\frac{1}{2}\,\mathrm{\bar H}_{\e}\, g|^2$. We then have by Cauchy-Schwarz and by \eqref{eq:step3-2} that 
\begin{align}
&\big\|  \mathrm{A}_{\e} -\frac{1}{2}\,\mathrm{\bar H}_{\e}\, g\big\|_{4}^2 \leq C \int_{\pa \e} |\nabla  \mathrm{A}_{\e}| \big|\mathrm{A}_{\e} -\frac{1}{2}\,\mathrm{\bar H}_{\e}\, g\big| + |\mathrm{H}_{\e}| \big| \mathrm{A}_{\e} -\frac{1}{2}\,\mathrm{\bar H}_{\e}\, g\big|^2 \, d \H^2 \nonumber\\
&\leq C \int_{\pa \e} |\nabla  \mathrm{A}_{\e}|\big| \mathrm{A}_{\e} -\frac{1}{2}\,\mathrm{\bar H}_{\e}\, g\big| + |\mathrm{H}_{\e}-\mathrm{\bar H}_{\e}| \big|\mathrm{A}_{\e} -\frac{1}{2}\,\mathrm{\bar H}_{\e}\, g \big|^2 +\mathrm{\bar H}_{\e}  \big|\mathrm{A}_{\e} -\frac{1}{2}\,\mathrm{\bar H}_{\e}\, g \big|^2  \, d \H^2 \nonumber \\
&\leq  C\Big(\|\nabla \mathrm{A}_{\e} \|_{2}\, \big\| \mathrm{A}_{\e} -\frac{1}{2}\,\mathrm{\bar H}_{\e}\, g \big\|_{2} +\| \mathrm{H}_{\e}-\mathrm{\bar H}_{\e} \|_{2}\, \big\| \mathrm{A}_{\e} -\frac{1}{2}\,\mathrm{\bar H}_{\e}\, g  \big\|_{4}^2+ \big\| \mathrm{A}_{\e} -\frac{1}{2}\,\mathrm{\bar H}_{\e}\, g \big\|_{2}^2\Big)\nonumber\\
&\leq C \bigg[\bigg(\!\Big\|  \mathrm{A}_{\e} -\frac{\mathrm{\bar H}_{\e}}{2}\, g\Big\|_{2}\!\!+ \| \mathrm{H}_{\e}-\mathrm{\bar H}_{\e} \|_{2}\!\bigg) \Big\|  \mathrm{A}_{\e} -\frac{\,\mathrm{\bar H}_{\e}}{2} g\Big \|_{4}^2\!\!+\bigg(\!1\!+\Big\|  \mathrm{A}_{\e} -\frac{\mathrm{\bar H}_{\e}}{2}\, g\Big\|_{2}\bigg)\Big\|  \mathrm{A}_{\e} -\frac{\,\mathrm{\bar H}_{\e}}{2}\, g\Big\|_{2} \bigg]. \nonumber
\end{align}
Since the estimates \eqref{e.small osc} and \eqref{ffeq8} imply 
\[
 \| \mathrm{H}_{\e}-\mathrm{\bar H}_{\e} \|_{L^2(\pa \e)} \! + \big\| \mathrm{A}_{\e} -\frac{1}{2}\,\mathrm{\bar H}_{\e}\, g \big\|_{L^2(\pa \e)} \leq C \big( \| \mathrm{H}_{\e}-\mathrm{\bar H}_{\e} \|_{L^2(\pa \e)} \! + \big\|\ringg{\mathrm{A}}_{\e} \big\|_{L^2(\pa \e)} \big)\leq C \bar \eps^{\frac{q}{2}},
\]
when $\bar \eps>0$  is chosen small enough, we may absorb the term with $L^4$-norm on the right hand side with the that one on left hand side and as a consequence, we get 
\begin{equation}
\label{eq:step3-4}
\| \mathrm{A}_{\e} -\frac{1}{2}\,\mathrm{\bar H}_{\e}\, g \|_{L^4(\pa \e)}^2 \leq C \bar \eps^{\frac{q}{2}}. 
\end{equation}
Therefore  using   \eqref{ffeq16}, \eqref{ffeq11},  and \eqref{eq:step3-4} the above implies 
\begin{equation}\label{ffeq18}
\|\mathrm{H}_{\e}\|_{L^4(\pa \e)}^4\leq C\|\mathrm{A}_{\e}\|_{L^4(\pa \e)}^4\leq C\big(\|\mathrm{A}_{\e}  -\frac{1}{2}\,\mathrm{\bar H}_{\e}\, g \|_{L^4(\pa \e)}^4+\mathrm{\bar H}_{\e}^4 P(\e)\big)\leq C.
\end{equation}

The inequality \eqref{ffeq18} and the perimeter bound imply that we may use the interpolation inequalities from \cite[Section 6]{ManCa} on $\pa F$. First, by \cite[Proposition 6.2]{ManCa}, we have
\begin{equation}\label{ffeq30}
\max_{\pa \e}|\mathrm{H}_{\e}-\mathrm{\bar H}_{\e}|\leq C\big(\|\nabla\mathrm{H}_{\e}\|_{L^4(\pa \e)}+\|\mathrm{H}_{\e}-\mathrm{\bar H}_{\e}\|_{L^4(\pa \e)}\big)\,.
\end{equation}
To estimate the right hand side we  apply \cite[Proposition 6.5]{ManCa} and obtain
\begin{align}
\|\nabla\mathrm{H}_{\e}\|_{L^4(\pa \e)}&\leq C \|\mathrm{H}_{\e}-\mathrm{\bar H}_{\e}\|_{W^{2,2}(\pa \e)}^{\frac{3}{4}}\, \|\mathrm{H}_{\e}-\mathrm{\bar H}_{\e}\|_{L^2(\pa \e)}^{\frac{1}{4}}, \label{ffeq28}\\
\|\mathrm{H}_{\e}-\mathrm{\bar H}_{\e}\|_{L^4(\pa \e)}&\leq C \|\mathrm{H}_{\e}-\mathrm{\bar H}_{\e}\|_{W^{1,2}(\pa \e)}^{\frac{1}{2}}\, \|\mathrm{H}_{\e}-\mathrm{\bar H}_{\e}\|_{L^2(\pa \e)}^{\frac{1}{2}}.\label{ffeq29}
\end{align}
Note that  the constants in \eqref{ffeq30}, \eqref{ffeq28} and \eqref{ffeq29} are uniform, in particular, independent of $\varepsilon$. 
Using \eqref{ffeq15},  \eqref{ffeq17} and  \eqref{ffeq18}  we have
\begin{align}
|\lambda|\leq C \label{ffeq19}\\
\|\nabla \mathrm{H}_{\e} \|_{L^2(\pa \e)}^2 \leq C.\label{ffeq20}
\end{align}
Therefore, in order to  prove 
\begin{equation}\label{ffeq27}
\|\mathrm{H}_{\e}-\mathrm{\bar H}_{\e}\|_{W^{2,2}(\pa \e)}\leq C,
\end{equation}
we need a uniformly bound for  $ \|\nabla^2 \mathrm{H}_{\e}\|_{L^{2}(\pa \e)}$.

To this aim, we proceed in the  spirit of \cite[Lemma 2.5]{KS01}. We multiply \eqref{ffeq21} by $\nabla  \mathrm{H}_{\e}$ and integrate by parts to have 
\[
\int_{\pa \e} |\nabla^2\mathrm{H}_{\e}|^2\, d \, d \H^2
=\int_{\pa \e} (\Delta  \mathrm{H}_{\e})^2\, d \H^2 +\int_{\pa \e} \mathrm{A}_{\e}*\mathrm{A}_{\e}*\nabla \mathrm{H}_{\e}*\nabla \mathrm{H}_{\e}\,d \H^2.
\]
Then, by \eqref{ffeq23}, arguing as in \eqref{eq:step3-2} and using \eqref{eq:step3-4} we have 
\begin{equation}\label{ffeq22}
\begin{split}
\int_{\pa \e} |\nabla^2\mathrm{H}_{\e}|^2\, d \, d \H^2 &\leq \int_{\pa \e} (\Delta  \mathrm{H}_{\e})^2\, d \H^2 + C\int_{\pa \e} |\mathrm{A}_{\e}|^2\,|\nabla \mathrm{H}_{\e}|^2\,d \H^2\\
&\leq  \int_{\pa \e} (\Delta  \mathrm{H}_{\e})^2\, d \H^2 +C\big( 1 +  \| \mathrm{A}_{\e} -\frac{1}{2}\,\mathrm{\bar H}_{\e}\, g \|_{L^4(\pa \e)}^2 \,\|\nabla \mathrm{H}_{\e}\|_{L^4(\pa \e)}^2\big)\\
&\leq  \int_{\pa \e} (\Delta  \mathrm{H}_{\e})^2\, d \H^2 +C\big( 1 +  \bar \eps^{\frac{q}{2}} \,\|\nabla \mathrm{H}_{\e}\|_{L^4(\pa \e)}^2\big).
\end{split}
\end{equation}
We use again the Euler-Lagrange equation \eqref{eq:eulerlagr} together with the inequalities  \eqref{ffeq11},  \eqref{ffeq19} and \eqref{ffeq16} and apply the Young's inequality and obtain 
\[
(\Delta  \mathrm{H}_{\e})^2 \leq C\big(1+|\mathrm{A}_{\e}|^6 \big) \qquad \text{on } \, \pa F. 
\]
This and  \eqref{eq:step1-2} yield
\begin{equation}\label{ffeq24}
\int_{\pa\e} (\Delta  \mathrm{H}_{\e})^2\, d \H^2\leq C\Big(1+\int_{\pa \e}|\mathrm{A}_{\e}|^6\,d \H^2\Big).
\end{equation}
We apply \cite[Proposition 6.1]{ManCa} with $p=1$ to the function $|\mathrm{A}_{\e}|^3$ and use  H\"{older}'s inequality to have 
\begin{align*} 
\left(\int_{\pa \e}|\mathrm{A}_{\e}|^6\, d \H^2\right)^{\frac12}&\leq C\Big(\int_{\pa \e}|\mathrm{A}_{\e}|^2\, |\nabla\mathrm{A}_{\e}|\, d \H^2+\int_{\pa \e}|\mathrm{A}_{\e}|^3\, d \H^2\Big)\\
&\leq C\Big(\|\mathrm{A}_{\e}\|_{L^4(\pa\e)}^2\,\|\nabla\mathrm{A}_{\e}\|_{L^2(\pa\e)}+\|\mathrm{A}_{\e}\|_{L^4(\pa\e)}^{\frac{3}{4}} \Big)
\end{align*}
Therefore, the inequality
\begin{equation}\label{ffeq25}
\|\nabla\mathrm{A}_{\e}\|_{L^2(\pa\e)}\leq C,
\end{equation}
which is a consequence of \eqref{eq:step3-2} with \eqref{eq:step3-4}, along with  \eqref{ffeq18} implies $\|\mathrm{A}_{\e}\|_{L^6(\pa\e)}^6\leq C$. Hence, we obtain
\begin{equation}\label{ffeq26}
\int_{\pa\e} (\Delta  \mathrm{H}_{\e})^2\, d \H^2\leq C.
\end{equation}

We apply \cite[Proposition 6.1]{ManCa} to the function $\sqrt{|\nabla \mathrm{H}_{\e}|^4+\delta^2}$,  send $\delta\to 0$ and by the Young's inequality with \eqref{ffeq20} we have  
\begin{equation*}
\|\nabla \mathrm{H}_{\e}\|_{L^4(\pa \e)}^2 \leq C \big(\| \nabla^2 \mathrm{H}_{\e}\|_{L^2(\pa \e)}^2 + \|\nabla \mathrm{H}_{\e}\|_{L^2(\pa \e)}^2 \big) \leq C \big(1 +\| \nabla^2 \mathrm{H}_{\e}\|_{L^2(\pa \e)}^2\big),
\end{equation*} 
which together with \eqref{ffeq22} and \eqref{ffeq26}~yields
\begin{equation*}
\| \nabla^2 \mathrm{H}_{\e}\|_{L^2(\pa \e)} \leq  C\big( 1 +  \bar \eps^{\frac{q}{2}} \,\| \nabla^2 \mathrm{H}_{\e}\|_{L^2(\pa \e)}\big).
\end{equation*}
Therefore, when $ \bar \eps$ is small enough we obtain $\| \nabla^2 \mathrm{H}_{\e}\|_{L^2(\pa \e)}  \leq C$. 
In conclusion, the estimate \eqref{ffeq27} holds.

We use \eqref{e.small osc}, \eqref{ffeq28}, \eqref{ffeq29}  and \eqref{ffeq27} to obtain 
\begin{equation*}
\|\nabla\mathrm{H}_{\e}\|_{L^4(\pa \e)}\leq C \bar \varepsilon^{\frac{1}{8}}\quad\quad\text{and}\quad\quad 
\|\mathrm{H}_{\e}-\mathrm{\bar H}_{\e}\|_{L^4(\pa \e)}\leq C \bar \varepsilon^{\frac{1}{4}}\,.
\end{equation*}
By \eqref{ffeq30} this  implies 
\begin{equation}\label{e.acca bar}
\max_{\pa \e}|\mathrm{H}_{\e}-\mathrm{\bar H}_{\e}|\leq C\bar\eps^{\frac{1}{8}}\,.
\end{equation}
Let us write \eqref{e.acca bar} in a slightly different way. We define 
\[
\mathrm{H}_{\e}^0=\frac{2P(\e)}{3|\e|}
\]
and notice that by \eqref{ffeq11},  \eqref{eq:step1-2} and $|F| = |B_1|$ it holds
\[
\mathrm{\bar H}_{\e} - C \bar \eps^{\frac{q}{2}} \leq 2 \leq \mathrm{H}_{\e}^0 \leq  \mathrm{\bar H}_{\e}  + C  \bar \eps^{\frac{q}{2}}. 
\]
Therefore we have by \eqref{e.acca bar} that  
\begin{equation}\label{e.acca bar bis}
\big|\mathrm{H}_{\e}-\mathrm{H}_{\e}^0\big| \leq  C  \bar \eps^{\frac{q}{2}}.
\end{equation}
This in turn implies
\begin{equation}
\delta_{\text{cmc}}(\e):=\Big\|\frac{\mathrm{H}_{\e}}{\mathrm{H}_{\e}^0}-1\Big\|_{C^0(\pa\e)}\leq C  \bar \eps^{\frac{q}{2}}.
\end{equation}

Notice that the quantity $\delta_{\text{cmc}}(\cdot)$ is scaling invariant, while $\mathrm{H}_{\e}^0$ is not.
Therefore, scaling $\e$ such that $\mathrm{H}_{\alpha \e}^0=2$ and possibly choosing a smaller $\bar\varepsilon>0$, the assumptions of \cite[Theorem 1.8]{KM} are satisfied by $\widetilde{\e}=\alpha \e$, where $\alpha=P(\e)/P(B_1)$.
Accordingly, by the proof of \cite[Theorem 1.8]{KM}  and the arguments in \cite{MoPoSpa}, there exists $\widetilde{w}\in C^{1,1}(\SSS^2)$ such that, up to  translation, it holds $\textup{bar}(\widetilde F)=0$ and $\pa \widetilde{\e}=\{(1+\widetilde{w}(x))x\,:\,x\in\SSS^2\}$ with $\|\widetilde{w}\|_{C^{1}(\SSS^2)}\leq C\,\delta_{\text{cmc}}(\e)$. Thus,
$$\pa \e=\{(1+w(x))x\,:\,x\in\SSS^2\}\quad\text{with}\quad w(x)=\frac{1}{\alpha}\Big[\Big(1-\alpha\Big)+\widetilde{w}\Big],$$
and 
 $$ \|w\|_{C^{1}(\SSS^2)}\leq |1-\alpha|+\|\widetilde{w}\|_{C^{1}(\SSS^2)}\leq C\bar\varepsilon^{\,\widetilde{q}}\,,$$
for some $\widetilde{q}\in (0,1)$, by using \eqref{eq:step1-2} and its consequence $\alpha^{-1}\leq 1$.
We may now use  the quantitative Alexandrov theorem from \cite{MoPoSpa} for nearly spherical sets, by possibly decreasing  $\bar\varepsilon>0$, which implies 
\begin{equation*}
P(\e) - P(B_1) \leq C \|w\|_{W^{1,2}(\SSS^2)}^2 \leq  C\|\mathrm{H}_{\e}-\mathrm{\bar H}_{\e} \|_{L^2(\pa \e)}^2.
\end{equation*}
Thus, we have obtained $P(\e) - P(B_1) \leq C_1 \|\mathrm{H}_{\e}-\mathrm{\bar H}_{\e} \|_{L^2(\pa \e)}^2$ for a uniform constant $C_1>0$. Therefore, when $\bar\varepsilon<1/C_1$, it holds
\begin{equation*}
J_\eps(\e) = \| \mathrm{H}_{\e}-\mathrm{\bar H}_{\e}  \|_{L^2(\pa \e)}^2 - \eps P(\e) \geq - \eps P(B_1) = J_\eps(B_1)
\end{equation*}
and the equality is true only when $\e$ is the ball. By the minimality of $\e$ we have equality in the above and thus $\e$ is the ball.  
\end{proof}

We are now ready to prove  Theorem \ref{t:LS}, which follows easily from Proposition \ref{t.main} as we will see~below.
\begin{proof}[Proof of Theorem \ref{t:LS}]
Let $\bar\varepsilon>0$ be from  Proposition \ref{t.main} and fix $\varepsilon\in (0,\bar\varepsilon)$.
For every $C^2$-regular  $E\subset\R^3$ satisfying the conditions \eqref{e.condition}, by Proposition \ref{t.main}, we know that $J_\varepsilon(E)\geq J_\varepsilon(B_1)$, which implies 
$$P(E)-P(B_1)\leq \varepsilon^{-1\,}\| \mathrm{H}_{E}-\mathrm{\bar H}_{E}  \|_{L^2(\pa E)}^2\,.$$
The rest of the statement of Theorem \ref{t:LS} follows  from Corollary \ref{l.fra1}.
\end{proof}

\begin{remark}
By a simple scaling argument we have the statement of Theorem~\ref{t:LS} for sets with generic  volume $|E| = v$. Indeed,  fix $v>0$ and denote the radius $\rho$ with $|B_\rho| = v$. By Theorem~\ref{t:LS} for every  $\delta_0>0$ there exists $C(\delta_0)>0$ such that for every $C^2$-regular set $E\subset \R^3$ with $|E|=|B_\rho|$ and  $P(E) \leq  (4\pi \sqrt[3]{2}  - \delta_0)\rho^2$ it holds 
\begin{equation} \label{eq:Lemma5.1-1}
\big|P(E)- 4 \pi \rho^2\big|\leq C \rho^2 \, \big\|\mathrm{H}_{E}-\mathrm{\bar H}_{E}\big\|_{L^2(\partial E)}^2.
\end{equation}
\end{remark}

\section{Application to volume preserving curvature flow}\label{sec:3bis}

We begin this section by proving the following consequence of Theorem~\ref{t:LS}. 

\begin{proposition}\label{prop:5-1}
Assume that there are  balls $B_{r}(x_1), \dots, B_{r}(x_N)$ which are  quantitatively disjoint   in the sense that $|x_i - x_j|\geq 2r +\delta_1$ with $i \neq j$ for some $\delta_1 >0$ such that their union $F = \bigcup_{i=1}^N B_{r}(x_i)$ has the measure $|F| = |B_1|$. Then there is $\eps_0= \eps_0(\delta_1, M)>0$ such that for all $C^2$-regular sets $E \subset \R^3$ with $|E|= |B_1|$, $P(E)\leq M$ and  $|E \Delta F|\leq \eps_0$ it holds
\[
P(E) - 4 \pi N^{\frac13}  \leq C \| \mathrm{H}_{E}-\mathrm{\bar H}_{E}\|_{L^2(\pa E )}^2,
\]
for a constant $C$ that depends on $M$ and $\delta_1>0$. 
\end{proposition}

\begin{proof}
We note that we may assume that  $\| \mathrm{H}_{E}-\mathrm{\bar H}_{E}\|_{L^2(\pa E )} \leq \eps_1$, for small $\eps_1>0$, since otherwise the claim is trivially true for  $C\geq M/\eps_1$. Then by Theorem \ref{thm:JN} and \eqref{ffeq1} there exists a union of disjoint equisize balls, denote it by  $\tilde F$, such that $|\tilde F| = |B_1|$ and
\[
\sup_{x \in E \Delta \tilde F} |d_{\tilde F}| + \big| P(E) - P(\tilde F) \big|\leq C \eps_1^q.  
\]
By the assumption $|E \Delta F|\leq \eps_0$, where $F = \bigcup_{i=1}^N B_r(x_i) $ with $|x_i-x_j|\geq 2 r + \delta_1$ for $i \neq j$, the sets $F$ and $\tilde F$ must both have $N$ many balls with the same radius $r$ and  we have  
\begin{equation} \label{eq:Lemma5.1-2}
\sup_{x \in E \Delta F} |d_{F}|  + \big| P(E) - 4 \pi N^{\frac13} \big| \leq C \eps_0,
\end{equation}
when $\eps_1$ is chosen small. Note that $r= N^{-\frac13}$ and $N$ is bounded by the perimeter bound on $E$. From the above we deduce that there are disjoint  sets $E_1, \dots, E_N$ such that $E = \bigcup_{i=1}^N E_i$ and $B_{r-C\eps_0}(x_i) \subset E_i \subset B_{r+C\eps_0}(x_i) $. Let us choose for every $i$ radius $r_i$ such that $|E_i| = |B_{r_i}|$. Then $r-C\eps_0 \leq r_i \leq r+C\eps_0$ and by the isoperimetric inequality 
\[
P(E_i) \geq 4 \pi r_i^2\geq 4 \pi (r- C \eps_0)^2 \geq  4 \pi N^{-\frac23} - C\eps_0.
\]
 On the other hand by \eqref{eq:Lemma5.1-2} it holds 
\[
P(E) = \sum_{i =1}^N P(E_i) \leq 4 \pi N^{\frac13} + C \eps_1^q.
\] 
Therefore it holds  $|P(E_i) - 4 \pi r_i^2| \leq C\eps_0$ for all $i = 1, \dots, N$ and \eqref{eq:Lemma5.1-1}  implies 
\[
P(E_i) - 4 \pi r_i^2 \leq C \|\mathrm{H}_{E_i}-\mathrm{\bar H}_{E_i}\|_{L^2(\pa E_i)}^2 \leq C \|\mathrm{H}_{E_i}-\mathrm{\bar H}_{E}\|_{L^2(\pa E_i)}^2. 
\]
Summing over $i$  yields 
\[
P(E) - 4 \pi \sum_{i=1}^N r_i^2 \leq C \sum_{i=1}^N \|\mathrm{H}_{E_i}-\mathrm{\bar H}_{E}\|_{L^2(\pa E)}^2  = C \|\mathrm{H}_{E}-\mathrm{\bar H}_{E}\|_{L^2(\pa E)}^2.  
\]
Recall that the radii $r_i$ are such that  $\sum_{i=1}^N |B_{r_i}| = |B_1|$, i.e.,  $\sum_{i=1}^N r_i^3 = 1$. Therefore it holds 
\[
4 \pi \sum_{i=1}^N r_i^2 \leq 4 \pi \left( \sum_{i=1}^N r_i^3  \right)^{\frac23} \left( \sum_{i=1}^N 1  \right)^{\frac13}  =4 \pi  N^{\frac13}
\]
which implies  the claim. 
\end{proof}

In this section we will show how the sharp quantitative Alexandrov theorem, i.e.,  Theorem \ref{t:LS}, implies the convergence of the flat flow for the volume preserving mean curvature flow \eqref{eq:VMCF}.   We begin by recalling the definition of the flat flow. This was first given in \cite{MSS}, following the associated scheme proposed in \cite{ATW, LS}, but here we use the variant given in \cite{Vesa} since it simplifies the forthcoming analysis. We refer to  \cite{MSS, MoPoSpa, Vesa} for a more detailed introduction.

We fix the time step $h>0$, and  given a bounded set of finite perimeter $E$ with $|E| = v\in (0,+\infty)$,  we consider the minimization problem 
\beq
\label{def:min-prob}
\min  \Big{\{} P(F) + \frac{1}{h}\int_F  d_E \, dx : \,  |F| = v \Big{\}}
\eeq
 and note that the minimizer exists but might not be unique.  Here $d_E$ denotes the signed distance function as defined in \eqref{def:signdist}. We define the dissipation of a set $F$ with respect to a set $E$ as
 \beq
\label{def:distance}
\mathcal{D}(F,E) := \int_{F \Delta E} \dist(x,\partial E)\, dx 
\eeq
and observe that we may write the minimization problem \eqref{def:min-prob} as 
\beq
\label{def:min-probbis}
\min  \Big{\{} P(F) + \frac{1}{h}\mathcal{D}(F,E)
 : \,  |F| = v \Big{\}}.
\eeq
We note that by a simple scaling argument we may reduce to the case $v = |B_1|$ by changing the value of $h$. 

Let $E(0) \subset \R^3$ then be a bounded set of finite perimeter with $|E(0)|= |B_1|$ and which coincides with its Lebesgue representative. 
We construct  the discrete-in-time evolution $\{E_k^{(h)}\}_{k\in \N}$ by recursion in such a way that $E(0)=E$ and assuming that $E_k^{(h)}$ is defined we set $E_{k+1}^{(h)}$ to be a minimizer of \eqref{def:min-prob} with $E= E_k^{(h)}$. Notice that by standard regularity theory $\partial E_k^{(h)}$ is $C^{2,\alpha}$-regular (for all $\alpha\in (0,1)$) and therefore in \eqref{def:min-probbis} we use this exact representative in order to compute $\dist(\cdot,\partial E_k^{(h)})$. We define the {\em approximate volume preserving flat flow}  $\{E^{(h)}(t)\}_{t \geq 0}$ by setting
\[
E^{(h)}(t) = E_k^{(h)} \qquad \text{for }\, t \in [kh, (k+1) h).
\]
Let us then recall some basic a priori estimates for the approximative flat flow $\{E^{(h)}(t)\}_{t \geq 0}$.  
Using the minimality of $E_{k+1}^{(h)}$ and comparing its energy \eqref{def:min-prob} against the previous set  $E_k^{(h)}$, we obtain the following important estimate   for all $k =0,1, \dots $
 \beq
\label{eg:energy-compa}
P(E_{k+1}^{(h)})  + \frac{1}{h}\mathcal{D}(E_{k+1}^{(h)},E_k^{(h)})   \leq P(E_{k}^{(h)})  . 
\eeq
In view that $\partial E_{k+1}^{(h)}$ is $C^{2,\alpha}$-regular (for all $\alpha\in (0,1)$), it satisfies the Euler-Lagrange equation
 \beq
\label{eg:Euler-Lag}
\frac{d_{E_k^{(h)}}}{h} = - \mathrm{H}_{E_{k+1}^{(h)}} + \lambda_{k+1}^{(h)} \qquad \text{on }\, \partial E_{k+1}^{(h)},
\eeq
in the classical sense, where $\lambda_{k+1}^{(h)}$ is the Lagrange multiplier due to the volume penalization. It follows from the argument in the proof of  \cite[Lemma 3.6]{MSS} or \cite[Lemma 4.2]{Vesa}, that 
 \beq
\label{eg:dissipation-0}
\int_{ \partial E_{k+1}^{(h)}} d_{E_{k}^{(h)}}^2 \, d \H^2 \leq C \mathcal{D}(E_{k+1}^{(h)},E_k^{(h)})
\eeq
for some universal constant $C>0$.
By summing \eqref{eg:energy-compa} from $k =0$ to $k \to \infty$, and  using \eqref{eg:Euler-Lag} and \eqref{eg:dissipation-0}, we obtain the dissipation inequality
 \beq
\label{eg:dissipation}
\int_{h}^\infty \int_{ \partial E^{(h)}(t)} (\mathrm{H}_{E^{(h)}(t)} - \mathrm{\bar H}_{E^{(h)}(t)})^2 \, d \H^2 dt \leq C P(E(0)) .
\eeq
Note also that \eqref{eg:energy-compa} implies that $t\mapsto P(E^{(h)}(t))$ is non-increasing.  
%
We now recall the definition of {\em flat flow}.
\begin{definition}\label{dff}
A flat flow solution of \eqref{eq:VMCF}  is any  family of sets $\{E(t)\}_{t \geq 0}$ which is a cluster point of $\{E^{(h)}(t)\}_{t \geq 0}$, i.e., 
\[
E^{(h_n)}(t) \to E(t) \quad \text{as } \, h_n \to 0 \quad \text{in } \, L^1 \quad \text{for almost every }\, t >0 .
\]
\end{definition}
By \cite[Theorem 1]{Vesa} there exists a flat flow starting from $E(0)$.

We are interested in the long time behavior of the flow. To this aim we recall the following elementary  algebraic lemma from  \cite{JuMoPoSpa}.

\begin{lemma}\label{lem:an1}
Let $K\in \N$ and $\{a_k\}_{k\in \{1,\ldots, K\}}$ be a sequence of non-negative numbers.
Assume  that there exists $C>1$ such that 
\[
\sum_{k=i}^{K}a_k\leq C a_i
\]
for every $i\in\{1,\ldots, K\}$.
Then,
$$
\sum_{k=i+1}^{K}a_k\leq \Bigl(1-\frac1C\Bigr)^{i}S
$$
for every $i\in \{1,\ldots, K-1\}$, where $S:=\sum_{k=1}^{K}a_k$.
\end{lemma}

The following lemma is a particular case of the well-known fact that $C^2$-sets are $\Lambda$-minimizers of the perimeter (see e.g. \cite[Lemma 4.1]{AFM}). 
\begin{lemma}\label{lem:eke}
Let $N \in \N$ and $B_r(x_1), \dots, B_r(x_N)$ be disjoint balls in $\R^n$, with $|x_i-x_j|\geq 2r+ \delta_1$ for $i\neq j$,  and denote $F =  \bigcup_{i=1}^N B_r(x_i)$. Then there exists  a constant $\Lambda>0$, which depends on $r,n,N$ and $\delta_1$, such that for every set of finite perimeter $E \subset \R^n$ it holds 
\[
P(F) \leq P(E) + \Lambda |E\Delta F|.
\]
\end{lemma}

\begin{proof}[\textbf{Proof of Theorem \ref{thm2:MCF}}]

Let $\{E^{h_n}(t)\}_{t\geq 0}$ an approximate flat flow converging to $\{E(t)\}_{t\geq 0}$ and let $N$, $r>0$, and $x_i(t)$, $i=1, \dots, N$, be as in the statement. By rescaling we may assume that $|E(0)| = |B_1|$. 
Set
\[
f_n(t) = P(E^{(h_n)}(t)).
\]
By \eqref{eg:energy-compa} the $f_n$'s are monotone non-increasing functions which are  bounded by $P(E(0))$. Therefore, by Helly's Selection Theorem,
up to passing to a further sub-sequence (not relabeled), we may assume that the functions $f_n$'s converge pointwise to some non-increasing function $f_\infty:[0,+\infty)\to \R$.  Note that it follows from \eqref{asym} and Lemma \ref{lem:eke} that 
\[
\lim_{t \to \infty}f_\infty(t) \geq 4 \pi N^{\frac13}.
\]
 We need to distinguish two cases.

\noindent{\bf Case 1.} We  first assume that 
\beq\label{case1}
f_\infty(t)> 4\pi N^{\frac13} \quad\text{for all $t\in [0, +\infty)$.}
\eeq

We obsere that by using the argument in \cite{JN}, i.e., using  the estimates (4.23) and (4.25)  in  the proof of  \cite[Theorem~1.1]{JN}, we have that  for every $\eps>0$ there exists $T_\eps>1$  with the following property:  for every $T>T_\eps$   there exists $h_0=h_0(\eps, T)>0$  such that for  $0<h_n\leq h_0$ and $t\in (T_\eps, T)$ there exist points $x^{(h_n)}_1(t)$, \dots,  $x^{(h_n)}_N(t)$ satisfying  
    \begin{itemize}
\item[(a)] $|x^{(h_n)}_i(t)-x^{(h_n)}_j(t)|\geq 2r$ and $x^{(h_n)}_i(t)\to x_i(t)$ as $h_n\to 0$,
\item[(b)] $\displaystyle \sup_{ E^{(h_n)}(t)\Delta F^{(h_n)}(t)}\dist(\cdot , \pa F^{(h_n)}(t))\leq \eps$, where we have set  $\displaystyle F^{(h_n)}(t):=\bigcup_{i=1}^NB_r\big(x^{(h_n)}_i(t)\big)$.
\end{itemize}
Recall now that $|E^{(h_n)}(\cdot )\Delta E(\cdot )|\to 0$ uniformly in $(T_\eps, T)$, as $E^{(h_n)}(\cdot )$ and $E(\cdot )$ are H\"{o}lder continuous \cite{MSS, JN,  Vesa}, and recall that \eqref{asym} holds.  Hence by choosing $T_\eps$ larger and $h_0$ smaller, we may also ensure  (by triangular inequality) that $F^{(h_n)}(t)$ is uniformly $(C\eps)$-close to $F(t)$ in $(T_\eps, T)$ with respect to the $L^1$-distance, and thus  with respect to the Hausdorff distance up to change the constant $C$, which is always independent both of $n$ and $T$. Hence, by choosing $\eps$ sufficiently small and recalling  \eqref{separated}, we may  assume that  in fact 
\[
|x^{(h_n)}_i(t)-x^{(h_n)}_j(t)|\geq 2r+\frac{\delta_1}2\quad \text{ for $0<h_n\leq h_0$ and $t\in (T_\eps, T)$.}
\]
We may thus use Lemma \ref{prop:5-1} to deduce that 
\begin{equation} \label{eq:thm2-0}
P(E^{(h_n)}(t)) - 4 \pi N^{\frac13}  \leq C \| \mathrm{H}_{E^{(h_n)}(t)}-\mathrm{\bar H}_{E^{(h_n)}(t)}\|_{L^2(\pa E^{(h_n)}(t) )}^2,
\end{equation}
for a constant that depends on $P(E(0))$ and $\delta_1$.

We proceed by proving that the $E(t)$ converges in $L^1$-sense. The argument is now similar to that of \cite[Theorem~1.2]{JuMoPoSpa} but  some modifications are needed. In what follows $C$ will denote a positive constant independent both of $n$ and $T$, which may change from line to line and also within the same display.

 Let $T>T_\eps$ be as  above.  Then, by \eqref{case1} and the definition of $f_\infty$ it follows that 
 \[
 P(E^{(h_n)}(t))>  4\pi N^{\frac13}\quad\text{for all }t\in [0, T]
 \]
 and for $0<h_n\leq h_0$ (by taking $h_0>0$ smaller if needed). 
In turn (for the same $n$'s), 
by iterated use of \eqref{eg:energy-compa} and then by \eqref{eq:thm2-0} we deduce for all $t\in (T_\eps+h_n, T-h_n)$
\[
\begin{split}
 \frac{1}{h_n}\sum_{k= \lfloor\frac{t}{h_n}\rfloor}^{\lfloor\frac{T}{h_n}\rfloor-1} \mathcal{D}(E_{k+1}^{(h_n)},E_{k}^{(h_n)})  & \leq   P(E^{(h_n)}(t)) -P(E^{(h_n)}(T)) < P(E^{(h_n)}(t)) - 4\pi N^{\frac13} \\
 &\leq \widetilde C  \| \mathrm{H}_{E^{(h_n)}(t)}-\mathrm{\bar H}_{E^{(h_n)}(t)}\|_{L^2}^2,
 \end{split}
\]
where $\lfloor a\rfloor$ denotes the integer part of $a >0$. 
 
We proceed by using the Euler-Lagrange equation \eqref{eg:Euler-Lag} and the estimate \eqref{eg:dissipation-0}  to get  
\begin{equation} \label{eq:thm2-1}
\begin{split}
\|\mathrm{H}_{E^{(h_n)}(t)}-\mathrm{\bar H}_{E^{(h_n)}(t)}\big\|_{L^2}^2 &\leq \|\mathrm{H}_{E^{(h_n)}(t)}- \lambda_{\lfloor \frac{t}{h_n}\rfloor}^{(h_n)} \big\|_{L^2}^2 \\
&=\frac{1}{h_n^2}\| d_{E^{(h_n)}(t-h_n)} \big\|_{L^2(\partial E^{(h_n)}(t))}^2 \leq \frac{C}{h_n^2}\mathcal{D}(E_{\lfloor\frac{t}{h_n}\rfloor}^{(h_n)},E_{\lfloor\frac{t}{h_n}\rfloor-1}^{(h_n)}).
\end{split}
\end{equation}
The two previous displays yield
\[
\frac{1}{h_n} \sum_{k=\lfloor\frac{t}{h_n}\rfloor-1 }^{\lfloor\frac{T}{h_n}\rfloor-1} \mathcal{D}(E_{k+1}^{(h_h)},E_{k}^{(h_n)})   \leq \frac{C+h_n}{h_n^2} \mathcal{D}(E_{\lfloor\frac{t}{h_n}\rfloor}^{(h_n)},E_{\lfloor\frac{t}{h_n}\rfloor-1}^{(h_n)})\leq 
\frac{C'}{h_n^2} \mathcal{D}(E_{\lfloor\frac{t}{h_n}\rfloor}^{(h_n)},E_{\lfloor\frac{t}{h_n}\rfloor-1}^{(h_n)})
\]
with $C'>0$ independent of $T$ and $n$.
Denoting $a_k = \frac{1}{h_n} \mathcal{D}(E_{k+1}^{(h_h)},E_{k}^{(h_n)})$ and recalling that
\[
\frac{1}{h_n} \sum_{k=\lfloor\frac{T_\eps}{h_n}] }^{\lfloor\frac{T}{h_n}\rfloor-1} \mathcal{D}(E_{k+1}^{(h_h)},E_{k}^{(h_n)})\leq   P(E^{(h_n)}(T_\eps))\leq P(E(0))
\]
we may use Lemma \ref{lem:an1} to deduce 
\begin{equation} \label{eq:thm2-2}
\frac{1}{h_n} \sum_{k=\lfloor\frac{t}{h_n}\rfloor }^{\lfloor\frac{T}{h_n}\rfloor-1} \mathcal{D}(E_{k+1}^{(h_h)},E_{k}^{(h_n)})   \leq P(E(0)) \left( 1 - \frac{h_n}{C'}\right)^{\lfloor\frac{t}{h_n}\rfloor-\lfloor\frac{T_\eps}{h_n}\rfloor}\leq C e^{-t/C'} 
\end{equation}
for all $t\in (T_\eps+h_n, T-h_n)$ and $0<h_n\leq h_0$.

By \cite[Proposition 3.4]{MSS} it holds
\[
 |E^{(h_n)}_{k+1}\Delta E^{(h_n)}_{k}| \leq C l P(E^{(h_n)}_{k+1})
+ \frac{C}{l} \mathcal{D}(E_{k+1}^{(h_n)},E_{k}^{(h_n)})  
\]
for all $l \leq \frac1C \sqrt{h_n}$.  Therefore, by the inequality above, by $P(E^{(h_n)}_{k+1}) \leq P(E(0))$ and by \eqref{eq:thm2-2}   for every $t ,s$, with  $T_\eps + h_n \leq t < s \leq t+1 \leq T-h_n$, we have that  
\beq\label{comesopra}
\begin{split}
|E^{(h_n)}(t)\Delta E^{(h_n)}(s)|  & \leq \sum_{k=\lfloor\frac{t}{h_n}\rfloor }^{\lfloor\frac{s}{h_n}\rfloor-1} |E^{(h_n)}_{k+1}\Delta E^{(h_n)}_{k}|\\
& \leq C P(E(0)) l \frac{s-t}{h_n} +\frac{C}{l}\sum_{k=\lfloor\frac{t}{h_n}\rfloor }^{\lfloor\frac{s}{h_n}\rfloor-1}\mathcal{D}(E^{(h_n)}_{k+1},E^{(h_n)}_{k})\\
&\leq  \frac{C  l}{h_n} +\frac{C \, h_n}{l} e^{-t/C'},
\end{split} 
\eeq
for all $l \leq \frac1C \sqrt{h_n}$ and $0<h_n\leq h_0$.
In particular, choosing $l= h_n e^{-t/2C'}$, we have
\[
|E^{(h_n)}(t)\Delta E^{(h_n)}(s)| \leq Ce^{-t/2C'}.
\]
Passing to the limit as $h_n\to 0$, we get
\beq  \label{eq:thm2-4bis}
|E(t)\Delta E(s)| \leq  Ce^{-t/2C'} \qquad \text{for all} \, \, T_\eps\leq t < s \leq t+1 \leq T.
\eeq
Since $T>T_\eps$ is arbitrary,  we conclude that $E(t)$ converges exponentially fast to a set of finite perimeter $F$  in $L^1$, which by our assumptions is necessarily  of the form  
\beq\label{Fform}
F = \bigcup_{i =1}^N B_r(x_i), \quad \text{with $|x_i -x_j| \geq 2r + \delta_1$ for  $i \neq j$.} 
\eeq
Let us also  observe that for   $t\in (T_\eps+h_n, T-h_n)$ and $0<h_n\leq h_0$ we have 
\beq \label{eq:thm2-4}
\begin{split}
\int_{t+h_n}^{T-h_n} \|\mathrm{H}_{E^{(h_n)}(s)} &-\mathrm{\bar H}_{E^{(h_n)}(s)}\big\|_{L^2(\partial E^{(h_n)}(s))}^2\, ds \leq \sum_{k= \lfloor\frac{t}{h_n}\rfloor}^{\lfloor\frac{T}{h_n}\rfloor-1} h_n \|\mathrm{H}_{E_{k+1}^{(h_n)}}-\mathrm{\bar H}_{E_{k+1}^{(h_n)}}\big\|_{L^2(\partial E_{k+1}^{(h_n)})}^2 \\
&\leq\frac{C}{h_n} \sum_{k=\lfloor\frac{t}{h_n}\rfloor }^{\lfloor \frac{T}{h_n}\rfloor-1} \mathcal{D}(E_{k+1}^{(h_h)},E_{k}^{(h_n)}) \leq C e^{-t/C'} \end{split}
\eeq
where  in the second inequality we argued as in  \eqref{eq:thm2-1}, while in the last we used \eqref{eq:thm2-2}. 

For the same $t>T_\eps+1$ and $0<h_n\leq h_0$  we deduce from \eqref{eq:thm2-4} and from the mean value theorem that there exists $t_n \in [t - e^{-t/2C'}, t]$ such that 
\[
\|\mathrm{H}_{E^{(h_n)}(t_n)} -\mathrm{\bar H}_{E^{(h_n)}(t_n)}\big\|_{L^2(\partial E^{(h_n)}(t_n))}^2 \leq C  e^{-t/2C'} .
\]
By Theorem \ref{thm:JN} and \eqref{ffeq1} we deduce that $E^{(h_n)}(t_n)$ is close to a set  $F^{(h_n)}(t_n):=\bigcup_{i=1}^N B_r(x_i^{(h_n)}(t_n) )$, for  suitable $x_i^{(h_n)}(t_n) \in \R^3$, not only in $L^1$-sense but also 
\beq  \label{eq:thm2-5}
\sup_{x\in E^{(h_n)}(t_n) \Delta F^{(h_n)}(t_n) }\dist(\cdot, \pa F^{(h_n)}(t_n)) + P(E^{(h_n)}(t_n)) - 4 \pi N^{\frac13} \leq C  e^{- qt/4C'}. 
\eeq
Since $t_n\leq t$ we have  $P(E^{(h_n)}(t_n)) \geq  P(E^{(h_n)}(t)) $ and thus \eqref{eq:thm2-5} holds a fortiori with  $P(E^{(h_n)}(t_n))$ replaced by $P(E^{(h_n)}(t))$. 
Therefore sending $h_n\to 0$,  we deduce by the lower semicontinuity of the perimeter that 
\beq\label{expP}
 P(E(t)) - 4 \pi N^{\frac13} \leq C  e^{- qt/4C'}.
\eeq
On the other hand, by Lemma~\ref{lem:eke} we have 
\beq\label{expP2}
4 \pi rN^{\frac13}= P(F)\leq P(E(t))+ \Lambda|E(t)\Delta F|
\eeq
with $F$ given by \eqref{Fform}. Recalling that $|E(t)\Delta F|\to 0$  exponentially fast, \eqref{expP} and \eqref{expP2} yield the exponential convergence of $P(E(t))$ to $N4 \pi r^2$ as $t\to+\infty$.

We now turn to the Hausdorff convergence. Notice that inequality \eqref{eq:thm2-5} is true for $t$, possibly with a worse exponential rate, by applying \cite[Lemma 4.3]{JN}, and observe that passing to the limit  in \eqref{eq:thm2-5} (up to a $t$-dependent further sub-sequence, if needed) yields
\beq\label{quasi}
\sup_{ E(t) \Delta F(t)} \dist(\cdot, \pa F(t))  \leq C  e^{- qt/2C'} ,
\eeq
where $F(t):=\bigcup_{i=1}^N B_r(x_i(t))$ for suitable $x_i(t) \in \R^3$. Recalling that
$|E(t)\Delta F|\to 0$ exponentially fast as $t\to+\infty$, we deduce that also $F(t)$ converges to $F$ exponentially fast (with respect to the $L^1$-distance and hence the Hausdorff distance) and thus \eqref{quasi} holds also with $F(t)$ replaced by $F$, possibly with a worse exponential rate.


\medskip

\noindent{\bf Case 2.} Assume that here exists $\bar t> 0$ such that $f_\infty(t) = 4 \pi N^{\frac13}$ for every $t\geq \bar t$. In this case we will  show that for large times $E(t)$ {\em coincides} with $F$. The argument is identical to that of \cite[Theorem~1.2-Case 2]{JuMoPoSpa} but reproduce it for the reader's convenience. 

By monotonicity of the functions $f_n$'s,  we have  that for every $T>\bar t$, $\{f_n\}$ converges uniformly to $f_\infty\equiv 4 \pi N^{\frac13}$ in $[\bar t, T]$.
In particular, using that
\[
\frac{1}{h_n} \mathcal{D}(E^{(h_n)}_{k+1}, E^{(h_n)}_{k}) \leq f_n(k h_n) - f_n((k+1) h_n),
\]
we deduce for every $t\in [\bar t , T-h_n]$
\[
\sum_{k=\lfloor\frac{t}{h_n}\rfloor}^{\lfloor\frac{T}{h_n}\rfloor-1}
h_n^{-1}\mathcal{D}(E^{(h_n)}_{k+1},E^{(h_n)}_{k})
\leq f_n\Big(\lfloor\frac{t}{h_n}\rfloor h_n\Big) - f_n\Big(\lfloor\frac{T}{h_n}\rfloor h_n\Big)=:b_n
\to 0 \qquad \textup{as } h_n\to 0.
\]
Arguing as in  \eqref{comesopra}, for every $\bar t  \leq t < s  \leq T-h_n$, we get
\begin{align*}
|E^{(h_n)}(t)\Delta E^{(h_n)}(s)|  \leq C l \frac{s-t}{h_n} P(E(0))+\frac{C}{l}\sum_{k=[\frac{t}{h_n}]  }^{\lfloor\frac{s}{h_n}\rfloor-1}\mathcal{D}(E^{(h_n)}_{k+1},E^{(h_n)}_{k}),
\end{align*} 
for all $l \leq \frac1C \sqrt{h_n}$. Choosing $\ell= \sqrt{b_n} h_n$, we conclude that
\begin{align*}
|E^{(h_n)}(t)\Delta E^{(h_n)}(s)|  \leq C \sqrt{b_n} (s-t) P(E(0))+C \sqrt{b_n} \to 0,
\end{align*}
that is $E(t) = E(s)$ for every $\bar t<t<s<T$. By the arbitrariness of $T>\bar t$ and recalling that by assumption $E(t)$ converges to a disjoint union of balls up to translations, we necessarily have $E(t)=F$ for all $t\geq \bar t$, with $F$ as in \eqref{Fform}.
\end{proof}

\section{The asymptotics of the 3D Mullins-Sekerka flow}\label{sec:4}

Let us first construct a flat flow solution for the Mullins-Sekerka flow \eqref{eq:mullins} in the three-dimensional flat torus $\T_R^3 = \R^3/(R\Z^3)$ with $R \geq 1$. The construction in the case of a bounded domain is due to Luckhaus and Sturzenhecker \cite{LS}, while the two dimensional flat torus $\T^2$ is considered in \cite{JuMoPoSpa}. The construction in $\T_R^3$  is  essentially  the same as in  the two-dimensional case but we repeat it for the reader's convenience.  First,  the perimeter of  $E $ in the flat torus $\T_R^3$  is defined as
\[
P_{\T_R^3}(E) := \sup \Big{\{} \int_E \div X \, dx : X \in C^1(\T_R^3,\R^3) , \, \| X\|_{L^\infty} \leq 1 \Big{\}}.
\]
Here $X \in C^1(\T_R^3,\R^2)$ means that the $\Z_R^3$-periodic extension of $X$ to $\R^3$ is continuously differentiable. For a given set of finite perimeter $E \subset \T_R^3$, with $|E|=v$, we consider the minimization problem 
\beq \label{def:min-prob2}
\min \Big{\{} P_{\T_R^3}(F) + \frac{h}{2} \int_{\T_R^3} |D \, U_{F,E}|^2 \, dx  :\quad \text{with }\, |F| = |E| = v \Big{\}},  
\eeq
where the function  $ U_{F,E}\in H^1(\T_R^3)$ is the solution of 
\beq \label{eq:potential1}
-\Delta U_{F,E} = \frac{1}{h} \left( \chi_F - \chi_E\right)
\eeq
with zero average. We note first that by a simple scaling argument we may reduce to the case $v = |B_1|$ by changing the values of $R$ and $h$.  As proven in \cite{LS, Rog} there exists a minimizer for \eqref{def:min-prob2}, but it might not be unique.  Concerning the regularity of the minimizers, we may  argue as in \cite{JuMoPoSpa} to deduce that  the minimizing set  $F$ is $C^{3,\alpha}$-regular and satisfies the associated Euler-Lagrange equation
\[
U_{F,E}  = -\mathrm{H}_{F} + \lambda \qquad \text{on }\, \pa F
\]
 in the classical sense, where  $\lambda$ is the Lagrange multiplier due to the volume constraint.

We denote 
 \beq
\label{def:distance2}
\mathfrak{D}(F,E) := \int_{\T_R^3} |D \, U_{F,E}|^2 \, dx ,
\eeq
where $U_{F,E}$ is defined in \eqref{eq:potential1}. We define the $H^{-1}$-norm of a function $f$ on the torus $\T_R^3$  by duality  as 
\[
\|f\|_{H^{-1}(\T_R^3)} :=  \sup \Big\{ \int_{\T_R^3} \varphi \, f \, dx : \|D \varphi\|_{L^2(\T_R^3)} \leq 1\Big\} .
\]
Then,  by   \eqref{eq:potential1} and  by  integration by parts we have
\beq
\label{def:Hmenouno-bound}
\|\chi_F - \chi_E\|_{H^{-1}(\T_R^3)}^2  \leq  h^2 \, \| D \, U_{F,E} \|_{L^2(\T_R^3)}^2 = h^2 \, \mathfrak{D}(F,E).
\eeq

We fix the time step $h >0$, the initial set $E(0) \subset \T_R^3$ and let $E_1^{(h)}$ be a minimizer of \eqref{def:min-prob2} with $E(0) =E$.   We construct the discrete-in-time evolution $(E_k^{(h)})_{k\in \N}$ as before by induction such that, assuming that $E_k^{(h)}$  is defined, we set $E_{k+1}^{(h)}$ to be a minimizer of \eqref{def:min-prob2} with $E= E_k^{(h)}$ and denote the associated potential for short  by $U_{k+1}^{(h)}$, which is the solution of  
\beq \label{eq:potential2}
-\Delta U_{k+1}^{(h)} = \frac{1}{h} \left( \chi_{E_{k+1}^{(h)}} - \chi_{E_{k}^{(h)}}\right)
\eeq 
with zero average. The Euler-Lagrange equation now reads as 
\beq \label{eq:Euler-Lagr2}
U_{k+1}^{(h)} = -\mathrm{H}_{E_{k+1}^{(h)}} + \lambda_{k+1}^{(h)} \qquad \text{on }\, \pa E_{k+1}^{(h)}.
\eeq
Using the minimality of $E_{k+1}^{(h)}$ and comparing its energy  against the previous set  $E_k^{(h)}$, we obtain 
\beq \label{eq:energy-compa2}
P_{\T_R^3}(E_{k+1}^{(h)})  +\frac{h}{2}  \mathfrak{D}(E_{k+1}^{(h)}, E_{k}^{(h)}) \leq P_{\T_R^3}(E_{k}^{(h)})  ,
\eeq
where $ \mathfrak{D}(E_{k+1}^{(h)}, E_{k}^{(h)})$ is defined in \eqref{def:distance2}.

As before we define the approximative flat flow  $\{E^{(h)}(t)\}_{t \geq 0}$ by setting
\[
E^{(h)}(t) = E_k^{(h)} \quad \text{and} \quad U^{(h)}(t) = U_k^{(h)} \quad \qquad  \text{for }\, t \in [kh, (k+1) h)
\]
and  we call a \emph{flat flow solution}  of \eqref{eq:mullins}  any cluster point  $\{E(t)\}_{t \geq 0}$ of $\{E^{(h)}(t)\}_{t \geq 0}$  as $h\to 0$,  i.e., 
\[
E^{(h_n)}(t) \to E(t) \quad \text{ in $L^1$ for almost every }\, t >0 \, \text{ and for some } \, h_n\to 0.
\]
Arguing exactly as in  \cite[Proposition 3.1]{Rog} we may conclude that there exists a flat flow  solution  of \eqref{eq:mullins}  starting from $E(0)$ such that $P_{\T_R^3}(E(t)) \leq P_{\T_R^3}(E(0))$, $|E(t)| = |E(0)|$ for every $t \geq 0$ and $\{E(t)\}_{t\geq 0}$ satisfies the equation \eqref{eq:mullins} in a weak sense. Moreover the estimate \eqref{eq:energy-compa2} implies the dissipation inequality, which states that for all $T_2>T_1$ it holds 
\beq \label{eq:dissipation2}
\frac12\int_{T_1+h}^{T_2+h} \|D\,  U^{(h)}(t)\|_{L^2(\T_R^3)}^2 \, dt  \leq P_{\T_R^3}(E_{T_1}^{(h)}) -  P_{\T_R^3}(E_{T_2}^{(h)}).
\eeq
Also the limit  $U^{(h)}(t) \to U(t)$, up to a sub-sequence, exists for almost every $t$.


We state the following continuity result for the flat flow. The proof can be traced from \cite{LS}, but we sketch it for the reader's convenience.  
\begin{lemma}
\label{lem:luckhaus}
Let $\{E^{(h)}(t)\}_{t \geq 0}$ be an approximative flat flow solution of \eqref{eq:mullins}  starting from a  set of finite perimeter $E(0) \subset \T_R^3$ with volume $|E(0)| = v$ and perimeter $P_{ \T_R^3}(E(0))\leq M$. Then for all  $ t>s\geq 2h$ with $h \leq t-s\leq1$ it holds
\[
|E^{(h)}(t) \Delta E^{(h)}(s)| \leq C (t-s)^{\frac14}, 
\]
where the constant depends on $v$ and $M$. 
\end{lemma}

\begin{proof}
 Let us fix $s,t$  as in the assumptions  and let $U^{(h)}(s)$ and $U^{(h)}(t)$ be the potentials associated with the sets $E^{(h)}(s)$ and $E^{(h)}(t)$.  We use \eqref{def:Hmenouno-bound} and  \eqref{eq:dissipation2}  for the approximative flat  flow to deduce  
\[
\begin{split}
\|\chi_{E^{(h)}(t)} -\chi_{E^{(h)}(s)}\|_{H^{-1}(\T_R^3)} &\leq \int_{s-h}^{t+h} \frac{1}{h}\|\chi_{E^{(h)}(\tau +h)} -  \chi_{E^{(h)}(\tau)}\|_{H^{-1}(\T_R^3)}\, d \tau\\
&\leq 2 \left( \int_{s-h}^{t+h}\frac{1}{h^2}\|\chi_{E^{(h)}(\tau +h)} -\chi_{ E^{(h)}(\tau)}\|_{H^{-1}(\T_R^3)}^2 \, d \tau\right)^{\frac12} \, \sqrt{t-s}\\
&\leq 2\left(\int_{h}^\infty \|D \,  U^{(h)}(t)\|_{L^2(\T_R^3)}^2 \, d \tau\right)^{\frac12} \, \sqrt{t-s} \\
&\leq 2\sqrt{2 P_{\T_R^3}(E(0))} \sqrt{t-s}. 
\end{split}
\]
By \cite[Lemma 3.1]{LS} the following interpolation inequality holds for $\varphi \in BV(\T_R^3)$ and all $\rho \in (0,\rho_0)$
\beq
\label{eq:interpolationLS}
\|\varphi\|_{L^1(\T_R^3)} \leq \rho (\|D \varphi\|_{L^1(\T_R^3)} +c) + \frac{C}{\rho}\|\varphi\|_{H^{-1}(\T_R^3)}.
\eeq
We use \eqref{eq:interpolationLS} with $\varphi = \chi_{E^{(h)}(t) } - \chi_{E^{(h)}(s) }$ and have by the above and by $P_{\T_R^3}(E^{(h)}(\bar t)) \leq   P_{\T_R^3}(E(0)) $ for all $\bar t >0$
\[
\begin{split}
|E^{(h)}(t) \Delta E^{(h)}(s)| &\leq \rho( P_{\T_R^3}(E(0))  + c) + \frac{C}{\rho} \|\chi_{E^{(h)}(t)} - \chi_{E^{(h)}(s)}\|_{H^{-1}(\T_R^3)}\\
&\leq  \rho( P_{\T_R^3}(E(0))  + c) + \frac{C}{\rho} \sqrt{2 P_{\T_R^3}(E(0))} \sqrt{t-s}.
\end{split}
\]
The claim follows by choosing $\rho = \rho_0(t-s)^{\frac14}$.
\end{proof}

To proceed, we need the following density estimate which is similar to \cite[ Proposition 4.1]{JuMoPoSpa}  and \cite[Lemma 2.1]{Sch}. 
\begin{lemma}
\label{lem:schatzle}
 Let $E\subset \T_R^3$, with $R\geq 1$, be a set of class $C^3$, with $|E| = v$ and  $P_{\T_R^3}(E) \leq M $,  and let $u_E \in C^1(\T_R^3)$  be a function with zero average such that  $\|D u_E\|_{L^2(\T_R^3)} \leq M$ and 
\beq \label{eq:boundary}
\mathrm{H}_E = -u_E +\lambda \quad \text{on } \,  \pa E \quad \text{ for some } \lambda\in\R .
\eeq
Then for all $\rho \in (0,1)$ it holds 
\beq \label{eq:density-est}
\frac{1}{C} \leq \frac{\H^2(\pa E \cap B_\rho(x))}{\rho^2} \leq C,
\eeq
where the constant $C>1$ depends only on $v$ and $M$.
\end{lemma}
\begin{proof}
We note that by \eqref{eq:boundary}   for every $X \in C^1(\T_R^3; \R^3)$  it holds
\beq \label{eq:1stVar}
\int_{\pa E} \div^T X \, d \H^2 = \int_{E} \div \big((-u_E + \lambda)X\big) \, dx. 
\eeq
Therefore the statement follows from \cite[Lemma 2.1]{Sch} once we bound the Lagrange multiplier $\lambda \in \R$. The challenge is that we need to bound $\lambda$ such that the bound is independent of $R$. 

To this aim we obtain by  \cite[Proposition 2.3]{JN},  or \cite[Lemma 2.1]{MoPoSpa}, that there is a radius $r \in (0,1)$  depending on $v,M$, and points $x_1, x_2 \in \T_R^3$ such that $|x_1 -x_2| > 4r$ and 
\begin{equation}
\label{eq:pointsx}
|E\cap B_r(x_1)| \geq \frac12 |B_r| \qquad \text{and} \qquad |E\cap B_r(x_2)| \leq \delta|B_r|, 
\end{equation}
where $\delta>0$ is a small number which choice will be clear later.  Let $\eta_\eps$ be the standard convolution kernel and choose $f_1 = \chi_{B_{3r/2}(x_1)} * \eta_\eps$ and $f_2 = 27 \chi_{B_{\frac{r}{2}}(x_2)} * \eta_\eps$ for $\eps = \frac{r}{4}$, and let $\zeta \in C^\infty(\T_R^3)$ be the solution of  
\begin{equation}
\label{eq:zeta}
\Delta \zeta = f_1 - f_2
\end{equation}
in $\T_R^3$ with zero average. 

 Let us show that 
\begin{equation}
\label{eq:boundzeta}
\|D \zeta\|_{L^2(\T_R^3)} \leq C,
\end{equation}
for a constant $C$, which is independent of $R$. To this aim we multiply \eqref{eq:zeta} by $\zeta$, integrate by parts and recall the homogeneous Poincar\'e inequality  in three-dimension
\[
\|u\|_{L^6(\T_R^3)} \leq C \|D  u\|_{L^2(\T_R^3)},
\]
which holds for all $u$ with zero average for a uniform constant $C$, and deduce
\[
\begin{split}
\|D \zeta\|_{L^2(\T_R^3)}^2 &= - \int_{\T_R^3} \Delta \zeta \, \zeta \, dx = -\int_{\T_R^3}  (f_1 - f_2) \, \zeta \, dx\\
&\leq \|f_1-f_2\|_{L^{\frac{6}{5}}(\T_R^3)} \|\zeta \|_{L^{6}(\T_R^3)} \leq C \|D \zeta\|_{L^2(\T_R^3)},
\end{split}
\] 
where the constant depends on $v$ and $M$. Hence we have \eqref{eq:boundzeta}.

Let us fix $x \in \T_R^3$. By standard estimates for the Poisson equation it holds  for $\alpha \in (0,1)$
\[
\| D^2 \zeta \|_{C^{\alpha}(B_1(x))} \leq C( \|D \zeta \|_{L^2(B_2(x))} +  \|f_1 - f_2\|_{C^{\alpha}(B_2(x))}), 
\]
where the constant depends only on $\alpha$.  Since $x$ is arbitrary, the bound  \eqref{eq:boundzeta}  implies 
 \begin{equation}
\label{eq:boundzeta2}
\|D^2 \zeta\|_{L^\infty(\T_R^3)} \leq C.
\end{equation}

We proceed by multiplying the equation \eqref{eq:boundary} by $D \zeta \cdot \nu_E$, where $\nu_E$ is the outer unit normal of $E$, use the divergence theorem and have 
\[
\big|\int_{\pa E} (-u_E + \lambda)(D \zeta \cdot \nu_E)  \, d \H^{2} \big| =\big| \int_{\pa E} \mathrm{H}_E (D\zeta \cdot \nu_E) \, d \H^{2} \big|=\big| \int_{\pa E} \text{div}^T D \zeta \, d \H^{2}\big| \leq C \|D^2 \zeta\|_{L^\infty}  P_{\T_R^3}(E) . 
\] 
On the other hand it holds 
\[
|\lambda|  \big|\int_{\pa E} (D \zeta \cdot \nu_E)  \, d \H^{2}\big|  = |\lambda| \big|\int_{E} \Delta \zeta \, dx\big| = |\lambda| \big|\int_{E} f_1 - f_2\, dx\big|.
\]
  It follows from the choice of the points $x_1$ and $x_2$ in  \eqref{eq:pointsx} that $ \int_{E} f_1 - f_2\, dx \geq c_0>0$, for $c_0$ depending on $v$ and $M$, when $\delta$ is small enough. Therefore we need yet to show 
 \begin{equation}
\label{eq:bound3}
\big| \int_{\pa E} u_E(D \zeta \cdot \nu_E)  \, d \H^{2}\big| \leq C
\end{equation}
in order to have the uniform bound $|\lambda| \leq C$. 

We use the divergence theorem (for $E$ and $\T_R^3 \setminus E$ ) and have 
\[
\big| \int_{\pa E} u_E(D \zeta \cdot \nu_E)  \, d \H^{2}\big| \leq  \frac12 \int_{\T_R^3} |u_E| |\Delta \zeta| + |D u_E| |D \zeta| \, dx.  
\]
By the assumption $\|D u_E\|_{L^2(\T_R^3)} \leq M$ and by \eqref{eq:boundzeta} we have 
\[
\int_{\T_R^3} |D u_E| |D \zeta| \, dx \leq \|D u_E\|_{L^2(\T_R^3)} \|D \zeta\|_{L^2(\T_R^3)} \leq C.
\] 
On the other hand, by the equation \eqref{eq:zeta} and by the homogeneous Poincar\'e inequality it holds
\[
 \int_{\T_R^3} |u_E| |\Delta \zeta|  \, dx  \leq \|u_E\|_{L^6(\T_R^3)} \|f_1 -f_2\|_{L^{\frac65}(\T_R^3)} \leq  C \|D u_E\|_{L^2(\T_R^3)}  \leq C.
\]
Hence we have \eqref{eq:bound3} and  the claim follows. 
\end{proof}

Since we have the density estimates we may use the trace theorem due to  Meyers-Ziemer \cite{MZ}. Indeed it follows from \cite[Theorem 4.7]{MZ} that if  $E$ is as in the statement of Lemma \ref{lem:schatzle},  and hence satisfies \eqref{eq:density-est}, then for every $\varphi \in C^1(\T_R^3)$  it holds 
\begin{equation}\label{eq:meyers-ziemer}
\big| \int_{\pa E} \varphi \, d \H^2 \big| \leq C' \|\varphi\|_{W^{1,1}(\T_R^3)}\,,
\end{equation}
with $C'$ depending on $M$ and $R$.
The following proposition then follows from Theorem \ref{thm:JN} and Lemma \ref{prop:5-1}.
\begin{proposition}
\label{prop:quantialex-ms}
Let $E \subset \T_R^3$ be $C^3$-regular set such that $|E|=|B_1|$,  $P_{\T_R^3}(E) \leq M$ and let $u_E \in C^1(\T_R^3)$ be a function with zero average which satisfies
\[
\mathrm{H}_E = -u_E +\lambda \quad \text{on } \,  \pa E \quad \text{ for some } \lambda\in \R.
\]
Then there exist $q \in (0,1)$ and $R_0= R_0(M)\geq 1$   such that when $R \geq R_0$, there are disjoint balls $B_{r}(x_1), \dots, B_{r}(x_N)$ with $r = N^{-\frac13}$ such that for $F = \bigcup_{i=1}^N B_r(x_i)$  it holds 
\[
|E \Delta F| + |P(E) - 4 \pi N^{\frac13}| \leq C \|D u_E\|_{L^2( \T_R^3)}^q,
\]
where $C>1$ depends on $M$ and $R$. 

If in addition the balls $B_{r}(x_1), \dots, B_{r}(x_N)$  are quantitatively disjoint in the sense that $|x_i - x_j|\geq 2r +\delta_1$ with $i \neq j$ for some $\delta_1 >0$, then there is $\eps= \eps(\delta_1, M)>0$ such that if $|E \Delta F|\leq \eps$, where $F = \bigcup_{i=1}^N B_r(x_i)$, then   the estimate improves as 
\[
P(E) - 4 \pi N^{\frac13}  \leq C \|D u_E\|_{L^2( \T_R^3)}^2,
\]
for a constant $C$ that depends on $M, R$ and $\delta_1>0$. 
\end{proposition}

\begin{proof}
In order to prove the first inequality we may assume that $\|D u_E\|_{L^2( \T_R^3)} \leq \eps_1$, where $\eps_1>0$ is a small number which choice will be clear later, since otherwise the claim is trivially true when we choose $F =B_1$  and the constant $C$  large enough.  Let us begin by proving that every component of the boundary $\pa E$ are quantitatively bounded. To be more precise if $\Sigma'$ is a component of $\pa E$ then we claim that it holds 
\begin{equation}
\label{eq:sec-5-prop1}
\text{diam}(\Sigma') \leq C
\end{equation}
for a constant that is independent of $R$. Indeed, this follows from the density estimate \eqref{eq:density-est} by the following well-known argument. We can choose points $x_1, \dots, x_{k_0} \in \Sigma'$ such that $\frac12\leq |x_{i} - x_j|$ for $i \neq j$ and $k_0 \geq \text{diam}(\Sigma') $. Then by the perimeter bound and  \eqref{eq:density-est} it holds 
\[
M \geq \H^{2}(\Sigma') \geq  \sum_{i=1}^{k_0} \H^{2}(\Sigma' \cap B_{\frac14}(x_i)) \geq \frac{k_0}{16C} \geq \frac{\text{diam}(\Sigma')}{16C}
\]
  and \eqref{eq:sec-5-prop1} follows. Thanks to the diameter bound \eqref{eq:sec-5-prop1} we may use the results from the previous sections when we choose $R_0$ large enough, i.e., $ R_0 \geq 2C \geq 2 \text{diam}(\Sigma')$. 

  We use \eqref{eq:meyers-ziemer}  for $\varphi = u_E^2$ and obtain
\[
\|u_E\|_{L^2(\pa E)}^2 \leq \|u_E^2 \|_{W^{1,1}(\T_R^3)} \leq C' \|D u_E\|_{L^2(\T_R^3)}^2,  
\]
where the last inequality follows from Poincar\'e inequality. The assumption $\mathrm{H}_E = -u_E +\lambda $ on $\pa E$  yields
\begin{equation}
\label{eq:sec-5-prop2}
 \|\mathrm{H}_E -  \mathrm{\bar H}_E\|_{L^2(\pa E)}^2 \leq  \|\mathrm{H}_E - \lambda \|_{L^2(\pa E)}^2  =  \|u_E\|_{L^2(\pa E)}^2 \leq C' \|D u_E\|_{L^2(\T_R^3)}^2.  
\end{equation}
Since $\|D u_E\|_{L^2(\T_R^3)} \leq \eps_1$, then it holds $ \|\mathrm{H}_E - \mathrm{\bar H}_E\|_{L^2(\pa E)} \leq \tau$, where $\tau>0$ is from Theorem \ref{thm:JN}, when $\eps_1$ is small enough.  Then Theorem \ref{thm:JN} with \eqref{ffeq1} implies 
\begin{equation}
\label{eq:sec-5-prop3}
\sup_{x \in E \Delta F} |d_F|  + |P_{\T_R^3}(E) - 4 \pi N^{\frac13}| \leq C \|\mathrm{H}_E - \mathrm{\bar H}_E\|_{L^2(\pa E)}^q \leq C' \|D u_E\|_{L^2(\T_R^3)}^q ,
\end{equation}
where $F $ is as in the statement.  Hence we have the first claim. 

In order to prove the second claim we again notice that we may assume $\|D u_E\|_{L^2( \T_R^3)} \leq \eps_1$.  The  claim then follows from \eqref{eq:sec-5-prop1}, \eqref{eq:sec-5-prop2} and from Proposition \ref{prop:5-1}. 
\end{proof}

We are now ready the prove Theorem \ref{thm3:mullins}. Since the argument for the first statement  is similar to \cite[Theorem 1.1]{JN} (see also \cite{FJM}), while the second stamen follows from a similar argument  as in the proof of Theorem \ref{thm2:MCF}, we only give the outline of the proof. 
\begin{proof}[\textbf{Proof of Theorem \ref{thm3:mullins}}]
Let $\{E^{(h_n)}(t)\}_{t\geq 0}$ be an approximative flat flow converging to $\{E(t)\}_{t\geq 0}$. As we discussed above, we may assume that $|E(0)| = |B_1|$.

Let us consider the first claim. We begin by using the dissipation inequality \eqref{eq:dissipation2} for $T_l = l^2- h$, where $l = 1,2, \dots$ and obtain by the mean value theorem a sequence of times $T_l^h \in [l^2,(l+1)^2]$ such that 
\[
\|D\,  U^{(h)}(T_l^h)\|_{L^2(\T_R^3)}^2 \leq \frac{1}{l}\int_{l^2}^{(l+1)^2} \|D\, U^{(h)}(t)\|_{L^2(\T_R^3)}^2 \, dt  \leq \frac{P_{\T_R^3}(E_0)}{l}.
\]
Denote the associated sets by $E_l^h = E^{(h)}(T_l^h)$. Then by Proposition \ref{prop:quantialex-ms} we find an index $l_0$, which is independent of $h$, such that for every $l \geq l_0$  there is $N_l^h \in \N$ such that for the union of disjoint balls with radius $r_{l,h}= (N_l^h)^{-\frac13}$,  $F_l^h = \bigcup_{i=1}^{N_l^h} B_{r_{l,h}}(x_i^l)$, it holds 
\begin{equation}
\label{eq:thm1.4-1}
|E_l^h \Delta F_l^h| + |P(E_l^h) - 4 \pi (N_l^h)^{\frac13}| \leq C\, l^{-\frac{q}{2}}.
\end{equation}
Since the perimeter is decreasing $P(E^{(h)}(T_{l+1}^h)) \leq P(E^{(h)}(T_{l}^h)) $, also the number of the balls $N_l^h $ is decreasing  for $l\geq l_0$. Therefore, by  passing to the converging  sub-sequence $(h_n)$, we obtain that the limit sequence $\{N_l\}_{l \geq l_0}$, where  $N_l^{h_n}  \to N_l$, is also monotone and thus it has a limit $N = \lim_{l \to \infty} N_l$ and there is $l_1 \geq l_0$ such that $N_l = N$ for all $l \geq l_1$. But since $N_l^{h_n}  \to N_l$ and $N_l^{h_n}$ are also monotone,  we deduce that for every $L > l_1$ it holds 
\begin{equation}
\label{eq:thm1.4-2}
N_l^{h_n} = N  \qquad \text{for all }\,  l_1 \leq l \leq L 
\end{equation}
when $h_n $ is small, depending on $L$.  

We obtain from \eqref{eq:thm1.4-1} and \eqref{eq:thm1.4-2} that  
\[
 |P(E_l^{h_n}) - 4 \pi N^{\frac13}| \leq C\, l^{-\frac{q}{2}} \qquad \text{for all }\, l_1 \leq l \leq L 
\]
when $h_n$ is small. But again since $t \mapsto P(E^{(h_n)}(t))$ is monotone,  we deduce that for every $\eps>0$ there is $T_\eps>1$, independent of $h_n$, such that  for all $t \in   (T_\eps-1, T)$ where $T > 2T_\eps$ is arbitrary it holds
\begin{equation}
\label{eq:thm1.4-22}
 |P(E^{(h_n)}(t)) - 4 \pi N^{\frac13}| \leq  \eps , 
\end{equation}
when $h_n\leq h_0(\eps,T)$.  Using this together with  the dissipation inequality \eqref{eq:dissipation2} we deduce
\[
\frac12\int_{T_\eps}^{T} \|D \, U^{(h_n)}(t)\|_{L^2(\T_R^3)}^2 \, dt  \leq 2 \eps. 
\]

Let us denote $I_\eps^{h_n} = \{ t \in (T_\eps, T): \|D\,  U^{(h_n)}(t)\|_{L^2(\T_R^3)}^2 \geq \sqrt{\eps}\}$. Then by the above it holds $|I_\eps^{h_n}| \leq 4\sqrt{\eps}$. Moreover by Proposition \ref{prop:quantialex-ms}  we obtain that for all $t \in  (T_\eps, T) \setminus I_\eps^{h_n}$ there is  a union of $N$-many disjoint balls with radius $r = N^{-\frac13}$, $F^{(h_n)}(t) = \bigcup_{i=1}^{N} B_{r}(x_i^{h_n}(t))$  such that 
\[
|E^{(h_n)}(t) \Delta F^{(h_n)}(t) |  \leq C \eps^{\frac{q}{4}}.
\]
We use the H\"older continuity   $|E^{(h_n)}(t) \Delta E^{(h_n)}(s)| \leq C |t-s|^{\frac14}$, for $s+h \leq t \leq 1$ from   Lemma \ref{lem:luckhaus} and the fact that the above holds for all $t  \in (T_\eps, T) \setminus I_\eps^{h_n}$ with $|I_\eps^{h_n}| \leq 4\sqrt{\eps}$  to deduce that  for all $t \in (T_\eps,T)$ it holds 
\begin{equation}
\label{eq:thm1.4-3}
|E^{(h_n)}(t) \Delta F^{(h_n)}(t) |  \leq C \eps^{\frac{q}{8}},
\end{equation}
where $F^{(h_n)}(t) = \bigcup_{i=1}^{N} B_{r}(x_i^{h_n}(t))$ for $r = N^{-\frac13}$. 
The convergence $|E(t) \Delta F(t) |\to 0 $ as $t \to \infty$ follows by passing \eqref{eq:thm1.4-3} to the limit $h_n \to 0$ and recalling that $T$ is arbitrary.

\smallskip
We proceed to prove the second claim. We assume still that  $\{E^{h_n}(t)\}_{t\geq 0}$ is an approximate flat flow converging to $\{E(t)\}_{t\geq 0}$ and let $N$, $r>0$, and $x_i(t)$, $i=1, \dots, N$, be as in the statement. As in the proof of Theorem \ref{thm2:MCF} we set
\[
f_n(t) = P(E^{(h_n)}(t))
\]
and observe that  the $f_n$'s are bounded  monotone non-increasing functions. Therefore  the functions $f_n$'s converge pointwise to some non-increasing function $f_\infty:[0,+\infty)\to \R$. It follows from \eqref{eq:thm1.4-22} that $\lim_{t \to \infty} f_\infty(t) = 4 \pi N^{\frac13}$. Again  we  distinguish two cases.

\noindent{\bf Case 1.} Assume first  $ f_\infty(t)> 4\pi N^{\frac13}$ for all $t\in [0, +\infty)$. 

We recall that it follows  from \eqref{eq:thm1.4-3} that for any $\eps>0$ there is $T_\eps>1$ such that for all $t \in (T_\eps, T)$ it holds 
\[
|E^{(h_n)}(t) \Delta F^{(h_n)}(t) |  \leq C \eps^{\frac{q}{8}},
\] 
 when  $h_n$ is small.  Here $F^{(h_n)}(t) = \bigcup_{i=1}^{N} B_{ r}(x_i^{h_n}(t))$ for $r =  N^{-\frac13}$ and $|x_i^{h_n}(t) - x_j^{h_n}(t)|> 2 r$. 
We  use that fact that $|E^{(h_n)}(\cdot)\Delta E(\cdot)| \to 0$ uniformly in $[T_\eps, T]$ to obtain 
\[
|E(t) \Delta \tilde  F(t) |  \leq C \eps^{\frac{q}{8}} \qquad \text{for all }\, t \in (T_\eps, T), 
\]
where $\tilde F(t) = \bigcup_{i=1}^{ N} B_{ r}(\tilde x_i(t))$. The assumption  $|E(t) \Delta  F(t) | \to 0$ as $t \to \infty$ for $ F(t) = \bigcup_{i=1}^{ N} B_{r}(x_i(t))$ with $|x_i(t) -x_j(t)|\geq 2 r +\delta_1$ implies  that, by possibly enlarging  $T_{\eps}$, we have 
\begin{equation}
\label{eq:thm1.5-0}
|E^{(h_n)}(t) \Delta  F(t) |  \leq C \eps^{\frac{q}{8}} \qquad \text{for all }\, t\in (T_\eps, T) 
\end{equation}
when $h_n$ is small. We may thus use the second inequality in Proposition \ref{prop:quantialex-ms} to deduce 
\[
P(E^{(h_n)}(t) ) - 4 \pi N^{\frac13} \leq C \|D\,  U^{(h_n)}(t)\|_{L^2(\T_R^3)}^2
\]
for all $t \in (T_\eps, T)$. We then use the assumption  $ f_\infty(t)> 4\pi N^{\frac13} $ for all for all $t\in [0, +\infty)$ to deduce that it holds 
\begin{equation}
\label{eq:thm1.5-1}
0\leq P(E^{(h_n)}(t) ) - 4 \pi N^{\frac13} \leq C \|D\,  U^{(h_n)}(t)\|_{L^2(\T_R^3)}^2
\end{equation}
for all $t \in (T_\eps, T)$, when $h_n$ is small.

We proceed by choosing $k \geq  \lfloor \frac{T_\eps}{h_n}\rfloor +1$ and $k_T = \lfloor \frac{T}{h_n}\rfloor-1$. We  use \eqref{eq:energy-compa2} iteratively and then  \eqref{eq:thm1.5-1}    to obtain
\[
\begin{split}
\frac{h_n}{2} \sum_{i =  k}^{k_T-1}  \mathfrak{D}(E_{i+1}^{(h_n)}, E_{i}^{(h_n)}) &\leq P_{\T_R^3}(E_{k}^{(h_n)})  -P_{\T_R^3}(E_{k_T}^{(h_n)})  \leq  P_{\T_R^3}(E_{k}^{(h_n)}) - 4 \pi N^{\frac13} \\
&\leq    C \|D\,  U^{(h_n)}_{k}\|_{L^2(\T_R^3)}^2 = C  \mathfrak{D}(E_{k}^{(h_n)}, E_{k-1}^{(h_n)}),
\end{split}
\]
where the equality follows from the definition \eqref{def:distance2} of the dissipation  $\mathfrak{D}(\cdot , \cdot)$. We now argue precisely as in the proof of \cite[Theorem 1.3]{JuMoPoSpa} (which in turn is similar to the argument in the proof of Theorem \ref{thm2:MCF})  and deduce that for every $t \in (T_\eps+h_n, T-h_n)$ it holds 
\[
h_n \sum_{i = \lfloor \frac{t}{h_n} \rfloor}^{\lfloor \frac{T}{h_n} \rfloor}  \mathfrak{D}(E_{i+1}^{(h_n)}, E_{i}^{(h_n)}) \leq C e^{- \frac{t}{C'}}
\]
for a constants $C, C'$, depending on $M, \delta_1$ and $R$. Then,  by \eqref{def:Hmenouno-bound} and by the above inequality we have that for $T_\eps+h_n <s <t< s+1 < T-h_n$ it holds
\[
\begin{split}
\|\chi_{E^{(h_n)}(t)} &- \chi_{E^{(h_n)}(s)}\|_{H^{-1}(\T_R^3)} \leq \sum_{i=\lfloor\frac{s}{h_n}\rfloor+1}^{\lfloor\frac{t}{h_n}\rfloor} \|\chi_{E^{(h_n)}_{i}} - \chi_{E^{(h_n)}_{i-1}}\|_{H^{-1}(\T_R^3)}  \\ 
&\leq \frac{\sqrt{t-s}}{\sqrt{h_n}} \Big( \sum_{i=\lfloor\frac{s}{h_n}\rfloor+1}^{\lfloor\frac{t}{h_n}\rfloor} \|\chi_{E^{(h_n)}_{i}} - \chi_{E^{(h_n)}_{i-1}}\|_{H^{-1}(\T_R^3)}^2\Big)^{\frac12} \\
&\leq  \frac{1}{\sqrt{h_n}} \Big( \sum_{i=\lfloor\frac{s}{h_n}\rfloor+1}^{\lfloor\frac{t}{h_n}\rfloor} h_n^2 \, \mathfrak{D}(E^{(h_n)}_{i},E^{(h_n)}_{i-1})\Big)^{\frac12}\\
&\leq C\, e^{-\frac{s}{2C'}},
\end{split}
\]
when $h_n$ is small. Using \eqref{eq:interpolationLS} with $\varphi = \chi_{E^{(h_n)}(t)} - \chi_{E^{(h_n)}(s)}$ and  $\rho = e^{-\frac{s}{4C'}}$, we obtain by the above that 
\[
|E^{(h_n)}(t) \Delta E^{(h_n)}(s)|  \leq C\, e^{-\frac{s}{4C'}},
\]
by possible increasing $C$. Passing $h_n \to 0$ yields  $|E(t) \Delta E(s)|  \leq C\, e^{-\frac{s}{4C_0}}$ for $T_\eps < s <t < s+1<T$. Since $T$ was arbitrary we deduce that $E(t)$ converges exponentially fast to a set $F$ which by   \eqref{eq:thm1.5-0} is of the form $ F= \bigcup_{i=1}^{ N} B_{r}(x_i)$ with $|x_i -x_j| \geq 2 r +\delta_1$, i.e., 
\[
|E(t) \Delta F|  \leq C\, e^{-\frac{t}{4C'}},
\]
for $C$, which depends on $E(0)$ and $\delta_1$.  The convergence of the perimeters
\[
|P(E(t)) -  4 \pi N^{\frac13} | \leq C e^{-\frac{t}{C}}
\] 
follows from the same argument as in the proof of Theorem \ref{thm2:MCF}. 

\noindent{\bf Case 2.} There exists $\bar t>0$ such $ f_\infty(t) = 4\pi N^{\frac13}$ for all $t\in [\bar t, +\infty)$. 
In this case we argue exactly as in Theorem \ref{thm2:MCF} to infer that $E(t) = F$ for every $t\geq \bar t$.
\end{proof}

\section*{Acknowledgments}
V.~J.~was supported by the Academy of Finland grant 314227.  F.~O. and E.~S.~have been supported by the ERC-STG grant 759229 {\sc HiCoS}. M.~M. 
has been supported by PRIN 2022 Project “Geometric Evolution Problems and Shape Optimization (GEPSO)”, PNRR Italia Domani, financed by European Union via the Program NextGenerationEU, CUP D53D23005820006.  F.O. and M.~M.  are  members of the Gruppo Nazionale per l’Analisi
Matematica, la Probabilit\`a e le loro Applicazioni (GNAMPA), which is part of the Istituto
Nazionale di Alta Matematica (INdAM).


\begin{thebibliography}{99}
\bibitem{Luck}
\textsc{S. Luckhaus}.
\emph{Solutions for the two-phase {S}tefan problem with the {G}ibbs-{T}homson law for the melting temperature}. European J. Appl. Math. \textbf{1} (1990), 101--111.
\end{thebibliography}
\end{document}